\documentclass[a4paper,oneside,11pt]{amsart}

\usepackage{geometry}
\usepackage[english]{babel}
\usepackage{amsmath}
\usepackage[utf8]{inputenc}
\usepackage[T1]{fontenc}
 
\usepackage{amsthm}
\usepackage{amssymb}
\usepackage[all]{xy}
\usepackage[,colorlinks=true,citecolor= DarkBlue,linkcolor=Black]{hyperref}
\usepackage[svgnames,table]{xcolor}
\usepackage{graphicx}
\usepackage{tikz,tkz-tab}

\theoremstyle{definition}
\newtheorem{definition}{Definition}[section]

\theoremstyle{plain}
\newtheorem{theorem}{Theorem}
\newtheorem*{theorem*}{Theorem}
\newtheorem{proposition}[definition]{Proposition}
\newtheorem*{proposition*}{Proposition}
\newtheorem{lemma}[definition]{Lemma}
\newtheorem*{lemma*}{Lemma}
\newtheorem{sublemma}[definition]{Sub-lemma}
\newtheorem*{sublemma*}{Sub-lemma}
\newtheorem{corollary}[definition]{Corollary}
\newtheorem{observation}{Observation}

\newtheorem{fact}{Fact}
\newtheorem*{fact*}{Fact}

\theoremstyle{remark}
\newtheorem{remark}[definition]{Remark}

\newcommand{\Oc}{\mathbf{O}}
\renewcommand{\H}{\mathbf{H}}
\newcommand{\C}{\mathbf{C}}
\newcommand{\R}{\mathbf{R}}

\newcommand{\gl}{\mathfrak{gl}}
\renewcommand{\sl}{\mathfrak{sl}}
\newcommand{\su}{\mathfrak{su}}
\renewcommand{\sp}{\mathfrak{sp}}

\newcommand{\co}{\mathfrak{co}}
\renewcommand{\a}{\mathfrak{a}}
\newcommand{\g}{\mathfrak{g}}
\newcommand{\h}{\mathfrak{h}}    
\renewcommand{\k}{\mathfrak{k}}    
\newcommand{\m}{\mathfrak{m}}

\newcommand{\so}{\mathfrak{so}}
\newcommand{\p}{\mathfrak{p}}
\newcommand{\s}{\mathfrak{s}}

\newcommand{\z}{\mathfrak{z}}
\newcommand{\f}{\mathfrak{f}}

\newcommand{\e}{\mathrm{e}}
\renewcommand{\d}{\mathrm{d}}

\DeclareMathOperator{\Isom}{Isom}
\DeclareMathOperator{\Ad}{Ad}  
\DeclareMathOperator{\ad}{ad}

\DeclareMathOperator{\Ker}{Ker}

\DeclareMathOperator{\Span}{Span}

\DeclareMathOperator{\Kill}{Kill}

\DeclareMathOperator{\GL}{GL}
\DeclareMathOperator{\SL}{SL}
\DeclareMathOperator{\CO}{CO}

\DeclareMathOperator{\PO}{PO}
\DeclareMathOperator{\PU}{PU}
\DeclareMathOperator{\SO}{SO}
\DeclareMathOperator{\SU}{SU}

\DeclareMathOperator{\Sp}{Sp}
\DeclareMathOperator{\Spin}{Spin}
\DeclareMathOperator{\Ima}{Im}

\DeclareMathOperator{\Rk}{Rk}
\DeclareMathOperator{\Conf}{Conf}
\DeclareMathOperator{\Tr}{Tr}
\DeclareMathOperator{\id}{id}
\DeclareMathOperator{\diag}{diag}
\DeclareMathOperator{\Gr}{Gr}

\renewcommand{\S}{\mathbf{S}}
\renewcommand{\H}{\mathbf{H}}
 
\newcommand{\Ein}{\mathbf{Ein}} 
\newcommand{\X}{\mathbf{X}}

\renewcommand{\epsilon}{\varepsilon}
\renewcommand{\geq}{\geqslant}
\renewcommand{\leq}{\leqslant}
\renewcommand{\hat}{\widehat}  
\newcommand{\hx}{\hat{x}}
\renewcommand{\bar}{\overline}

\title[Conformal actions of simple Lie groups of rank $1$]{Conformal actions of real-rank $1$ simple Lie groups on pseudo-Riemannian manifolds}

\author{Vincent Pecastaing}

\begin{document}

\maketitle

\vspace{-.5cm}

\begin{center}
{\small
University of Luxembourg, Mathematics Research Unit \\
6, rue Richard Coudenhove-Kalergi L-1359 Luxembourg 
\footnote{
The author also acknowledges support from U.S. National Science Foundation grants DMS 1107452, 1107263, 1107367 "RNMS: Geometric Structures and Representation Varieties" (the GEAR Network).
}
}
\end{center}

\vspace{.4cm}

\begin{center}
\today
\end{center}

\begin{abstract}
Given a simple Lie group $G$ of rank $1$, we consider compact pseudo-Riemannian manifolds $(M,g)$ of signature $(p,q)$ on which $G$ can act conformally. Precisely, we determine the smallest possible value for the index $\min(p,q)$ of the metric. When the index is optimal and $G$ non-exceptional, we prove that $g$ must be conformally flat, confirming the idea that in a ``good'' dynamical context, a geometry is determined by its automorphisms group. This completes anterior investigations on pseudo-Riemannian conformal actions of semi-simple Lie groups of maximal real-rank \cite{zimmer87,bader_nevo,frances_zeghib}. Combined with these results, we obtain as corollary the list of semi-simple Lie groups without compact factor that can act on compact Lorentzian manifolds. We also derive consequences in CR geometry via the Fefferman fibration.
\end{abstract}

\tableofcontents

\section{Introduction}

Let $(M,g)$ be a pseudo-Riemannian manifold of signature $(p,q)$. When $p+q \geq 3$, the conformal class $[g] = \{e^{\sigma} g, \ \sigma \in \mathcal{C}^{\infty}(M)\}$ is a \textit{rigid} geometric structure on $M$ \cite{gromov88}. The group $\Conf(M,g)$ of diffeomorphisms preserving $[g]$ is then a Lie transformation group. Thus, under dynamical hypothesis, $\Conf(M,g)$ may have only few possible forms and furthermore, the metric also may have a prescribed form. In Riemannian geometry, the Ferrand-Obata theorem is a typical illustration of these principles \cite{ferrand71,ferrand96,obata71}, see also \cite{matveev} and \cite{schoen,webster} in other rigid geometric structures. This article is a contribution confirming this idea in other signatures, when $M$ is compact and $\Conf(M,g)$ contains simple Lie subgroups of rank $1$.

\vspace{.2cm}

Dynamics of semi-simple Lie groups on rigid geometric structures is a rich subject on which we have now many remarkable results, among which we may cite \cite{zimmer86,gromov88,zeghib97,kowalski,quiroga06,bader_frances_melnick}. The situation is well understood when the geometry naturally defines a finite invariant measure, but is less clear in the non-unimodular case, conformal dynamics being a good example of which.

Of course, the real-rank of the acting group has a central role. In a general context \cite{zimmer87}, Zimmer bounded the rank of semi-simple Lie groups without compact factor that act on compact manifolds by preserving a $G$-structure. In the case of a conformal structure (\textit{i.e.} $G = \CO(p,q)$, say with $p \leq q$), he obtained that the rank of the acting group is at most $p+1$. This upper bound is optimal since $\PO(p+1,q+1)$ is the conformal group of the\textit{ Einstein Universe }$\Ein^{p,q}$. Recall that this space is a smooth quadric in $\R P^{p+q+1}$ obtained by projectivizing the light-cone of $\R^{p+1,q+1}$. It naturally carries a conformal class of signature $(p,q)$, whose conformal group coincides with $\PO(p+1,q+1)$. In \cite{bader_nevo}, Bader and Nevo proved that when the group is simple and precisely has rank $p+1$, then it is locally isomorphic to $\PO(p+1,k)$, with $p+1 \leq k \leq q+1$. Finally, Frances and Zeghib proved that in the same context of maximal rank, $(M,g)$ is conformal to a quotient of the universal cover of $\Ein^{p,q}$ \cite{frances_zeghib}.

\vspace{.2cm}

These works describe very well the situation when a semi-simple Lie group of \textit{maximal} real-rank acts, but not for smaller values of the rank. In this article, we propose new contributions in the ``opposite'' situation of simple groups of rank $1$, \textit{i.e.} when the rank is the smallest possible. Our original motivation is the understanding of semi-simple Lie groups acting on compact Lorentzian manifolds, which - by the above results - is mainly reduced to the rank $1$ case. Since the approach we use naturally applies to other signatures, and has consequences in Levi non-degenerate CR geometry, we have chosen to highlight mainly the pseudo-Riemannian version of our results. This possibly opens a work leading to a complete description of semi-simple conformal groups of compact pseudo-Riemannian manifolds, that we leave for further investigations.

\subsection{Main results}

Up to local isomorphism, the list of rank $1$ simple Lie groups is $\SO(1,k)$, $\SU(1,k)$, $\Sp(1,k)$, with $k \geq 2$, and the real form $F_4^{-20}$ of $F_4$ (sometimes noted $FII$). We call \textit{index} of a pseudo-Riemannian metric the dimension of its maximal isotropic subspaces. Intuitively, metrics with large index admit lots of possible group actions since there are ``large spaces'' with no constraints. The question we ask here is, given a group $G$ of rank $1$, what is the \textit{smallest possible index} of a compact pseudo-Riemannian manifold on which it can act conformally. This will determine in some sense the indices compatible with the structure of $G$. When $G$ acts conformally on $(M,g)$ whose index is optimal, we naturally expect that there is not a large choice for the geometry of $(M,g)$. Except in the case of $F_4^{-20}$, this will be confirmed by our results. 

\vspace{.2cm}

Since $\SO(1,k)$ acts conformally on the Möbius sphere $\S^{k-1}$, the smallest possible index for this group is $0$, and by Ferrand-Obata theorem, this is the only possible example of a conformal Riemannian action of this group. By the same result, $\SU(1,k)$ \textit{cannot} act on a compact Riemannian manifold, but it acts on the Lorentzian Einstein Universe $\Ein^{1,2k-1}$, so the smallest possible index is $1$. Moreover, by Corollary 1 of \cite{article4}, when a Lie group locally isomorphic to $\SU(1,k)$ acts conformally on a compact Lorentzian manifolds, this manifold must be some quotient of the universal cover of $\Ein^{1,n-1}$. 

\vspace{.2cm}

Thus, our question concerns rank $1$ symplectic groups and the real form of $F_4$. For the first ones, the answer is given by the following

\begin{theorem}
\label{thm:symplectic}
Let $(M,g)$ be a compact pseudo-Riemannian manifold of signature $(p,q)$, with $p+q \geq 3$. Assume that $M$ admits a conformal action of $\Sp(1,k)$, with $k \geq 2$. Then,
\begin{itemize}
\item $\min(p,q) \geq 3$;
\item And if $\min(p,q)=3$, then $p+q \geq 4k+2$ and $(M,g)$ is conformally flat.
\end{itemize}
\end{theorem}

Thus, $\Sp(1,k)$ never appears in the conformal group of any compact pseudo-Riemannian manifold of signature $(1,n)$ or $(2,n)$, $n \geq 2$. There is a natural embedding $\Sp(1,k) \hookrightarrow \SO(4,4k)$, and the corresponding conformal action on $\Ein^{3,4k-1}$ is transitive. This $\Sp(1,k)$-homogeneous space is in fact a principal bundle $\Sp(1) \rightarrow \Ein^{3,4k-1} \rightarrow \partial \mathbb{H}_{\H}^k$ 
over the boundary of the quaternionic symmetric space. Thus, the bounds of the previous result are sharp. In fact, the proof shows that when $\Sp(1,k)$ acts on a pseudo-Riemannian manifold of signature $(3,n)$, then any (if any) compact minimal invariant subset is a single orbit conformally equivalent to this homogeneous space fibering over $\partial \mathbb{H}_{\H}^k$ (compare with \cite{nevo_zimmer}).

\vspace{.2cm}

Concerning $F_4^{-20}$, we obtain the following result.

\begin{theorem}
\label{thm:exceptional}
Let $(M,g)$ be a compact pseudo-Riemannian manifold of signature $(p,q)$, with $n = p+q \geq 3$. Assume that $M$ admits a conformal action of the exceptional simple Lie group $F_4^{-20}$ of real-rank $1$. Then $\min(p,q) \geq 9$.
\end{theorem}

The lower bound in the previous theorem is sharp, and it is achieved once more when the group acts on some Einstein Universe. We do not know if the metric must be conformally flat when it has index $9$, because our methods are clearly more difficult to manipulate with this group.

\vspace{.2cm}

All real forms of $F_4$ can be realized as derivation algebras of traceless \textit{Albert algebras} - also called \textit{exceptional simple Jordan algebras} - over real Cayley algebras of dimension $8$ (\cite{tomber}). In the case we are interested in, let $\Oc$ denote the classical algebra of octonions, and let $p$ be the diagonal $3 \times 3$ matrix $\diag(-1,1,1)$. Following \cite{tomber}, we define $\mathfrak{J}(\Oc,p)$ to be the set of ``$p$-Hermitian'' $3 \times 3$ matrix $x$ with coefficients in $\Oc$, \textit{i.e.} satisfying $x = p x^* p^{-1}$, where $x^*$ denotes the transconjugate of $x$. Explicitly, $\mathfrak{J}(\Oc,p)$ is formed by the matrices
\begin{equation*}
x =
\left ( \!\!\!
\begin{matrix}
\quad \xi_1 & c_1 & c_3 \\
-\bar{c_1} & \xi_2 & c_2 \\
-\bar{c_3} & \bar{c_2} & \xi_3
\end{matrix}
\right )
, 
\text{ with } \xi_1,\xi_2,\xi_3 \in \R, \text{ and } c_1,c_2,c_3 \in \Oc.
\end{equation*}
Then, we endow $\mathfrak{J}(\Oc,p)$ with the product $x \star y = \frac{1}{2} (xy +yx)$ which makes it a commutative, non-associative real algebra, and we can state

\begin{theorem*}[\cite{tomber}, Th.6]
The Lie algebra $\f_4^{-20}$ is isomorphic to the Lie algebra of derivations of $\mathfrak{J}(\Oc,p)$.
\end{theorem*}

Any derivation of this Albert algebra is skew-symmetric with respect to the quadratic form $q(x,y) = \Tr(x \star y)$ and preserves the $26$-dimensional subspace $\mathfrak{J}_0 = \{x \in \mathfrak{J}(\Oc,p) \ | \ \Tr x = 0\}$. Moreover, $q$ is non-degenerate and its restriction to $\mathfrak{J}_0$ has signature $(10,16)$ and $\mathfrak{J}_0$ is an irreducible subrepresentation of $\f_4^{-20}$, proving the existence of an embedding $\f_4^{-20} \hookrightarrow \so(10,16)$. Finally, $F_4^{-20}$ acts conformally on the Einstein Universe of signature $(9,15)$.


\subsection{Semi-simple conformal groups of compact Lorentzian manifolds} 

Since the works of Zimmer in \cite{zimmer86}, we know that if a semi-simple Lie group acts isometrically on a compact Lorentzian manifold, then it is locally isomorphic to $\SL(2,\R) \times K$, where $K$ is a compact semi-simple Lie group. The main example is the case where $\SL(2,\R)$ acts by left multiplication on a quotient $\SL(2,\R) / \Gamma$, where $\Gamma$ is a uniform lattice. Indeed, the Killing metric of $\SL(2,\R)$ being bi-invariant and Lorentzian, it induces a left-invariant Lorentzian metric on any quotient $\SL(2,\R) / \Gamma$. In fact, this example is essentially the only irreducible one that we can produce (\cite{gromov88}).

\vspace{.2cm}

One of our motivations in this work is to extend this description to conformal actions. Let $G$ be a semi-simple Lie group without compact factor acting conformally on a compact Lorentzian manifold. By \cite{zimmer87}, $G$ has real-rank at most $2$. Theorems \ref{thm:symplectic} and \ref{thm:exceptional} implies that it cannot admit $\Sp(1,k)$ nor $F_4^{-20}$ as direct factor. In the case $\Rk_{\R}(G) = 2$, the manifold is conformal to a quotient of the universal cover of ${\Ein}^{1,n-1}$ by Theorem 1.5 of \cite{bader_frances_melnick}, proving that $G$ must be locally isomorphic to a subgroup of $\SO(2,n)$. So we just have to consider actions of $\SO(1,k)$ and $\SU(1,k)$ on compact Lorentzian manifolds to obtain: 

\begin{theorem}
\label{thm:classification_lorentz}
Let $(M^n,g)$, $n \geq 3$, be a compact Lorentzian manifold and $G$ be a semi-simple Lie group without compact factor. Assume that $G$ acts conformally on $M$. Then, $G$ is locally isomorphic to one of the following groups:
\begin{enumerate}
\item $\SO(1,k)$, $2 \leq k \leq n$ ;
\item $\SU(1,k)$, $2 \leq k \leq n/2$ ;
\item $\SO(2,k)$, $2 \leq k \leq n$ ;
\item $\SO(1,k) \times \SO(1,k')$, $k,k' \geq 2$ and $k+k' \leq \max(n,4)$.
\end{enumerate}
Conversely, for any $n \geq 3$, every group in this list acts conformally on a compact Lorentzian manifold of dimension $n$, namely $\Ein^{1,n-1}$.
\end{theorem}

In fact, this theorem can be directly deduced from the main result of \cite{article4}. However, the approach we suggest here is more direct and does not involve a lot of geometric considerations, and it might be easier to generalize it to other signatures or geometric structures (see below).

\subsection{Relations with CR geometry: The Fefferman fibration}

As consequences of Theorems \ref{thm:symplectic} and \ref{thm:exceptional}, we obtain results on automorphisms groups of CR structures, via a generalization of the Fefferman fibration due to \v Cap and Gover \cite{cap_gover08}. Consider $M^{2n+1}$ a compact manifold endowed with a partially integrable almost CR structure of hypersurface type, whose Levi form is non-degenerate of signature $(p,q)$, $p+q = n$, also equipped with an appropriate line bundle (see Section 2.3 of \cite{cap_gover08}, this is not restrictive in the embedded case). Then, Theorem 2.4 of \cite{cap_gover08} asserts that there is a natural circle bundle $\tilde{M} \rightarrow M$, and that $\tilde{M}$ carries a natural conformal class of pseudo-Riemannian metrics $[g]$ of signature $(2p+1,2q+1)$. The correspondence $M \mapsto (\tilde{M},[g])$ being fonctorial, the CR automorphisms group of $M$ lifts to a subgroup of $\Conf(\tilde{M},g)$. The Fefferman bundle being compact when $M$ is compact, we can deduce from Theorems \ref{thm:symplectic} and \ref{thm:exceptional}:

\begin{corollary}
Let $M$ be a compact manifold endowed with a partially integrable almost CR structure of hypersurface type, whose Levi form has signature $(p,q)$, and admitting a line bundle as above.

\begin{itemize}
\item If $\Sp(1,n)$, $n \geq 2$, acts by CR automorphisms on $M$, then $\min(p,q) \geq 1$ and if $\min(p,q) = 1$, then $M$ is flat as a CR structure.
\item If $F_4^{-20}$ acts by CR automorphisms on $M$, then $\min(p,q) \geq 4$.
\end{itemize}
\end{corollary}

\begin{remark}
In the first situation, the fact that $\min(p,q) \geq 1$ immediately follows from Schoen-Webster theorem. Flatness means that the corresponding Cartan geometry is flat, \textit{i.e.} the CR structure defines an atlas of $(G,\X)$-manifold on $M$, with $G = \PU(2,N)$ and $\X = \{|z_1|^2 + |z_2|^2 = |z_3|^2 + \cdots + |z_{N+2}|^2\} \subset \C P^{N+1}$.
\end{remark}

\begin{remark}
In the second case, we do not know if this lower bound is optimal.
\end{remark}

Combined with Theorem 1.5 of \cite{bader_frances_melnick}, we finally deduce an analogue of Theorem \ref{thm:classification_lorentz} for CR manifolds.

\begin{corollary}
Let $M^n$ be a compact manifold endowed with a partially integrable almost CR structure of hypersurface type, whose Levi form has Lorentzian signature, and admitting a line bundle as above. Assume that a semi-simple Lie group $G$ without compact factor acts on $M$ by CR automorphisms. Then,
\begin{itemize}
\item If $\Rk_{\R}(G) = 2$, then $G$ is locally isomorphic to a Lie subgroup of $\SU(2,n)$;
\item If $\Rk_{\R}(G) = 1$, then $G$ is isomorphic to some $\SO(1,k)$, $\SU(1,k)$ or $\Sp(1,k)$.
\end{itemize}
Conversely, every group in this list acts on a compact Levi-Lorentzian CR structure.
\end{corollary}

\begin{remark}
Contrarily to the case of compact Lorentzian manifolds, we do not know if the CR structure has to be flat when a non-compact simple Lie group act on it.
\end{remark}

\subsection{Organization of the article}

The central idea of the main results is that when $G$ acts on a compact manifold whose index is optimal, then any compact minimal $G$-invariant subset is a compact orbit conformal to an Einstein Universe. This dynamical property is based on methods going back to the 1980's \cite{zimmer86,zimmer87,gromov88}, but which were mainly developed in the case of unimodular geometric structures. More recent works of Bader and Nevo \cite{bader_nevo} use these ideas in conformal geometry, and they were also developed in the framework of Cartan geometries in \cite{bader_frances_melnick}. 

We start Section \ref{s:background} by giving details on these properties. Then, we establish some general facts and observations about conformal actions in optimal index. Quickly, we will obtain lower bounds on the index of the metric (Proposition \ref{prop:lower_bound}). These bounds will be sharp except in the case of $F_4^{-20}$, which for technical reasons is more difficult to handle. However, this proposition will give the first half of Theorem \ref{thm:symplectic} and almost completely prove Theorem \ref{thm:classification_lorentz}. What will be left is a control of the size of the possible acting group in terms of dimension of the manifold, which will follow from Section \ref{s:so1k} and Lemma \ref{lem:signature_optimal}.

Section \ref{s:conformal_flatness} is devoted to geometric considerations in optimal index, and proves the second half of Theorem \ref{thm:symplectic}. Even if we are mainly interested in actions of $\Sp(1,k)$ in signature $(3,n)$, our proof literally applies to Lorentzian actions of $\SU(1,k)$, so we present both of them in a unified way - even if the case of $\SU(1,k)$ already follows from \cite{article4}.

At last, Section \ref{s:f4} determines the optimal index for actions of $F_4^{-20}$. Proposition \ref{prop:lower_bound} of Section \ref{s:background} implies that it is at least $7$, so we will be left to prove that there does not exist actions in index $7$ or $8$, and Theorem \ref{thm:exceptional} will be established.

\subsection*{Conventions and notations}

Throughout this paper, $M$ always denote a smooth compact manifold, with $\dim M \geq 3$. If $G$ is a Lie group acting by diffeomorphisms on $M$, then its Lie algebra is identified with a Lie subalgebra of $\mathfrak{X}(M)$ via $X \in \g \mapsto \left \{ \left . \frac{\d}{\d t} \right |_{t=0} e^{-tX}.x \right \}_{x \in M}$.

Particularly, given a conformal action of $G$ on a pseudo-Riemannian manifold $(M,g)$, we will implicitly identify $\g$ with a Lie algebra of conformal vector fields of $M$. If $V \subset \g$ is a vector subspace and $x \in M$, we will note $V(x) = \{X_x, \ X \in V\} \subset T_xM$ the ``evaluation'' of $V$ at $x$. For instance, $\g(x)$ will denote $T_x(G.x)$.

If $\g$ is semi-simple and $\lambda$ is a restricted root, the notation $X_{\lambda}$ will implicitly mean that $X_{\lambda}$ belongs to the restricted root-space of $\lambda$. 

\subsection*{Acknowledgements} \textit{This work was realized during a stay at the University of Maryland. I would like to thank the members of the Mathematics Department for their hospitality. I am especially grateful to Jeffrey Adams for his clarifications about representation theory.}

\section{Background and general observations}
\label{s:background}

Our main results rely essentially on a proposition characterizing some particular orbits in every compact invariant subset of the manifold. It is based on the framework of ideas developed by Zimmer in the case of actions of semi-simple groups on finite volume $G$-structures. We already used it in a restricted situation in \cite{article3,article4}. We start this section by giving a more general formulation.

\subsection{Identifying singular orbits}

Let $G$ be a semi-simple Lie group. Let $\g = \a \oplus \m \oplus \bigoplus_{\lambda \in \Delta} \g_{\lambda}$ be the restricted root-space decomposition with respect to a Cartan decomposition $\g = \k \oplus \p$, $\a$ being a maximal Abelian subspace in $\p$ and $\m$ the centralizer of $\a$ in $\k$. We will always note $\theta$ the corresponding Cartan involution. Fix an ordering on the set of restricted roots $\Delta$ and denote by $S$ the connected Lie subgroup of $G$ whose Lie algebra is
\begin{equation*}
\s = \a \oplus \bigoplus_{\substack{\lambda \in \Delta \\ \lambda > 0}} \g_{\lambda}.
\end{equation*}

Let $G$ act conformally on a pseudo-Riemannian manifold $(M,g)$ of dimension at least $3$. If $x \in M$, we note $\g_x = \{X \in \g \ | \ X(x) = 0\}$ the Lie algebra of the stabilizer of $x$. Differentiating the orbital map $G \rightarrow G.x$, we obtain a natural identification $T_x(G.x) \simeq \g / \g_x$, so that $\g / \g_x$ inherits a quadratic form $q_x$ from the ambient metric $g_x$.

As it is done in \cite{article3}, the following proposition can be derived from a Cartan geometry version of Zimmer's embedding theorem, which is formulated in \cite{bader_frances_melnick}. In fact, it is almost stated in \cite{bader_nevo}, and can be recovered by considering a natural $G$-equivariant map $M \rightarrow \Gr(\g) \times \mathbb{P}(\g^* \otimes \g^*)$, and applying a version of Borel's density Theorem (see \cite{bader_nevo}, proof of Theorem 1). The important point is that the Zariski closure of $\Ad_{\g}(S)$ in $\GL(\g)$ does not admit proper algebraic cocompact subgroup.

\begin{proposition}
\label{prop:virtual_isotropy}
Assume that $S$ preserves a probability measure $\mu$ on $M$. Then, for $\mu$-almost every $x \in M$, we have
\begin{enumerate}
\item $\Ad(S) \g_x \subset \g_x$ ;
\item The induced action of $\Ad(S)$ on $\g / \g_x$ is conformal with respect to $q_x$.
\end{enumerate}
\end{proposition}

\begin{remark}
Note that the adjoint action of the stabilizer $G_x$ on $\g / \g_x$ is conjugated to its isotropy representation on $T_x(G.x)$. So, for any point $x$, $\Ad(G_x)$ acts on $\g / \g_x$ by preserving the conformal class $[q_x]$. Thus, this proposition gives conditions ensuring the existence of points where $S$ is ``virtually'' in the isotropy representation.
\end{remark}

\begin{remark}
\label{rem:orbit_closure}
Since $S$ is amenable, the existence of $\mu$ is guaranteed by the one of a compact $S$-invariant subset in $M$. It will be automatic when $M$ is compact. Notably, for every $x \in M$, there will exist $x_0 \in \bar{G.x}$ in which the conclusions of the proposition are valid.
\end{remark}

\subsection{$\Ad(S)$-invariant subalgebras of $\g$}

Our work consists in analyzing orbits of points given by Proposition \ref{prop:virtual_isotropy}. Let us first see what can be said about their isotropy.

\vspace{.2cm}

If $\h \subset \g$ is a vector subspace such that $\ad(\a) \h \subset \h$, then $\h = (\h \cap \g_0) \oplus \bigoplus_{\lambda \in \Delta} (\h \cap \g_{\lambda})$. This simply follows from the existence of $H \in \a$ such that the $\lambda(H), \ \lambda \in \Delta$ are non-zero and pairwise distinct. The next lemma exploits the invariance by the adjoint action of positive root-spaces and will be the starting point of our analyze.

\begin{lemma}
\label{lem:transversality}
Let $\lambda \in \Delta$ be such that $2\lambda \notin \Delta$ and let $\h$ be a subalgebra of $\g$ which is $\ad(\g_{\lambda})$-invariant. Then, 
\begin{itemize}
\item Either $\h \cap \g_{-\lambda}= 0$,
\item Or $\g_{-\lambda} \subset \h$.
\end{itemize}
\end{lemma}

\begin{proof}
This lemma is based on the following formula:
\begin{equation*}
\forall X,Y \in \g_{-\lambda}, \ [[X,\theta Y],X] = 2|\lambda|^2 B_{\theta}(X,Y)X - |\lambda|^2 B_{\theta}(X,X)Y,
\end{equation*}
where as usual, $B$ is the Killing form of $\g$, $B_{\theta}$ is the positive definite quadratic form $-B(\theta X, Y)$ and $|.|$ refers to the Euclidean structure on $\a^*$ provided by $B$. The proof of this fact is inspired from \cite{knapp}, Proposition 6.52(a) The idea is here to decompose $[X,\theta Y] \in \g_0$ according to $\g_0 = \a \oplus \m$. It is simply:
\begin{equation*}
[X,\theta Y] = \underbrace{\frac{1}{2}([X,\theta Y] - [\theta X, Y])}_{\in \a} + \underbrace{\frac{1}{2}([X,\theta Y] + [\theta X , Y])}_{\in \m}.
\end{equation*}
The component on $\a$ can be determined explicitly. We can compute that for all $H \in \a$,
\begin{equation*}
[[X,\theta Y] - [\theta X, Y],H] = B(X,[\theta Y,H]) + B(Y,[\theta X, H]) = -2\lambda(H) B_{\theta}(X,Y).
\end{equation*}
Consequently, $[X,\theta Y] - [\theta X, Y] = -2B_{\theta}(X,Y) H_{\lambda}$ where $H_{\lambda} \leftrightarrow \lambda$ under the dual identification given by the Killing form on $\a$. Now, let us note $Z = [[X,\theta Y],X]$. We have:
\begin{align*}
2Z & = [-2B_{\theta}(X,Y) H_{\lambda},X] + [[X,\theta Y],X] + [[\theta X, Y],X] \\
   & = 2B_{\theta}(X,Y) \lambda(H_{\lambda})X + Z + [[\theta X, Y],X].
\end{align*}
For the last term of the sum, since $[\g_{-\lambda},\g_{-\lambda}] = 0$, the Jacobi relation gives $[[\theta X, Y],X] = [[\theta X,X],Y]$. Since we have $[\theta X,X] = B_{\theta}(X,X) H_{\lambda}$ (take $X=Y$ in the previous computation), we finally obtain
\begin{equation*}
Z = 2|\lambda|^2 B_{\theta}(X,Y)X - |\lambda|^2 B_{\theta}(X,X)Y.
\end{equation*}

Now, the Lemma is immediate: if there exists a non-trivial $X \in \h \cap \g_{-\lambda}$, then for all $Y \in \g_{-\lambda}$, $[X,\theta Y] \in \h$ since $\theta Y \in \g_{\lambda}$. So, $[[X,\theta Y],X] \in \h$ and by the formula we have proved, we obtain $Y \in \h$ since $B_{\theta}(X,X) > 0$.
\end{proof}

\subsection{General observations for rank $1$ conformal actions in optimal index}
\label{ss:general_observations}

From now on, we restrict our attention to conformal actions of rank $1$ Lie groups. We collect here arguments and observations that will be used several times in the study of conformal actions of Lie groups of type $(BC)_1$.

\vspace{.2cm}

In this section, $\g$ is a Lie algebra isomorphic to either $\su(1,k)$, $\sp(1,k)$ or $\mathfrak{f}_4^{-20}$ (with $k \geq 2$). Let $(M,g)$ be a compact pseudo-Riemannian manifold of signature $(p,q)$ on which a Lie group $G$ with Lie algebra $\g$ acts conformally, with $p+q \geq 3$. We assume that $\min(p,q)$ is minimal with this property, \textit{i.e.} that there does not exist $(N,h)$ compact of signature $(p',q')$ with $\min(p',q') < \min(p,q)$ on which $G$ acts conformally.

\begin{observation}
\label{obs:fixed_points}
The group $G$ acts without fixed point and the action is essential, \textit{i.e.} $\forall g' \in [g]$, $G \nsubseteq \Isom(M,g')$. 
\end{observation}

\begin{proof}
If $x$ was a fixed point of $G$, then the isotropy representation $G \rightarrow \GL(T_x M)$ would be orthogonal with respect to $g_x$. By Thurston's stability theorem (see for instance Section 2 of \cite{cairns_ghys}), this representation would be faithful unless $G$ acted trivially on a neighborhood of $x_0$. By rigidity, a conformal map that acts trivially on an open subset is trivial. So, the isotropy representation would provide an embedding $G \hookrightarrow SO(p,q)$. Therefore, $G$ would also act locally faithfully on the Einstein space $\Ein^{p-1,q-1}$, contradicting the minimality of $\min(p,q)$.

If the action was inessential, then $G$ would act isometrically with respect to a metric of finite volume and signature $(p,q)$. Zimmer's embedding theorem (\cite{zimmer86}, Theorem A) would then yield an inclusion $\g \hookrightarrow \so(p,q)$, leading to the same contradiction as above.
\end{proof}

In any event, the restricted root-space decomposition of $\g$ has the form
\begin{equation*}
\g = \g_{-2\alpha} \oplus \g_{-\alpha} \oplus \g_0 \oplus \g_{\alpha} \oplus \g_{2\alpha}.
\end{equation*}
Denote by $\theta$ the Cartan involution adapted to this decomposition and $\g = \k \oplus \p$ the Cartan decomposition. We note $\g_0 = \a \oplus \m$ with $\a= \g_0 \cap \p$ the corresponding Cartan subspace and $\m  = \g_0 \cap \k$ the centralizer of $\a$ in $\k$. Except in the case of $\f_4^{-20}$, the compact Lie algebra $\m = \z_{\k}(\a)$ splits into $\m^1 \oplus \m^2$ and $\m^1 = \k \cap [\g_{-2\alpha},\g_{2\alpha}]$.
\begin{equation*}
\begin{array}{|c|c|c|c|c|c|}
\hline 
\g  &  \m & \m^1 & \m^2 & \dim \g_{\pm \alpha} & \dim \g_{\pm 2\alpha} \\
\hline
\su(1,k) & \m^1 \oplus \m^2 & \mathfrak{u}(1) & \su(k-1) & 2(k-1) & 1 \\
\hline
\sp(1,k) & \m^1 \oplus \m^2 & \sp(1) & \sp(k-1) & 4(k-1) & 3 \\
\hline 
\f_4^{-20} & \so(7) & & & 8 & 7 \\
\hline
\end{array}
\end{equation*}

\vspace{.2cm}

Let $\s$ be the Lie subalgebra $\s = \a \oplus \g_{\alpha} \oplus \g_{2\alpha}$ and $S < G$ be the corresponding Lie subgroup. Let $x \in M$ be a point in which the conclusions of Proposition \ref{prop:virtual_isotropy} are true.

\vspace{.2cm}

By Observation \ref{obs:fixed_points}, $x$ cannot be fixed by all the elements of $G$ and $\g_{x}$ is a proper Lie subalgebra of $\g$, which is $\Ad(S)$-invariant. Recall that
\begin{equation*}
\g_{x} = \bigoplus_{\lambda \in \{0, \pm \alpha, \pm 2\alpha\}} (\g_{x} \cap \g_{\lambda}).
\end{equation*}
The important starting point of the analysis of the orbit of $x$ is the following

\begin{observation}
\label{obs:transversality}
We have $\g_{x} \cap \g_{-2\alpha}=0$. 
\end{observation}

\begin{proof}
Indeed, if not we would have $\g_{-2\alpha} \subset \g_{x}$ by Lemma \ref{lem:transversality}. Since $\a \subset [\g_{2\alpha},\g_{-2\alpha}]$, we would get $\a \subset \g_{x}$, and it would follow that $\g_{\alpha} \oplus \g_{2\alpha} \subset \g_{x}$. Since $\g_{-\alpha} = [\g_{\alpha},\g_{-2\alpha}]$, we would obtain $\g = \g_{x}$ contradicting Observation \ref{obs:fixed_points}.
\end{proof}

Let $H \in \a$ be such that $\alpha(H) = 1$ and $h^t$ the adjoint action of $\Ad(e^{tH})$ on $\g / \g_{x}$, which is conformal with respect to $q_{x}$. Recall that we note $\pi_{x} : \g \rightarrow \g / \g_{x}$ the natural projection. The action of $h^t$ on $\g / \g_{x}$ easily gives a lot of orthogonality relations.

\begin{observation}
\label{obs:orthogonality}
Assume that we have $\lambda, \mu \in \{0, \pm \alpha, \pm 2\alpha\}$ and $X_{\lambda} \in \g_{\lambda}$, $X_{\mu} \in \g_{\mu}$ such that $b_{x}(\pi_{x}(X_{\lambda}), \pi_{x}(X_{\mu})) \neq 0$. Then, for any $\lambda', \mu'$ such that $\lambda'(H) + \mu'(H) \neq \lambda(H) + \mu(H)$, the spaces $\pi_{x}(\g_{\lambda'})$ and $\pi_{x}(\g_{\mu'})$ are orthogonal with respect to $b_{x}$.
\end{observation}

\begin{proof}
Since the action of $h^t$ on $\g / \g_{x}$ is conformal, we have $c \in \R$ such that $(h^t)^* b_{x} = e^{c t} b_{x}$. Our hypothesis gives $c = \lambda(H) + \mu(H)$. Therefore, for any $\lambda', \mu'$, and $u \in \pi_{x}(\g_{\lambda'})$, $v \in \pi_{x}(\g_{\mu'})$, the identity $e^{ct} b_{x}(u,v) = b_{x}(e^{\lambda'(H)t}u, e^{\mu'(H)t}v)$ implies $b_{x}(u,v) = 0$ as soon as $\lambda'(H) + \mu'(H) \neq \lambda(H) + \mu(H)$.
\end{proof}

We can now prove the following proposition which is our first significant result.

\begin{proposition}
\label{prop:lower_bound}
Assume that $\min(p,q) \leq \dim \g_{2\alpha}$. Then, $\g_{-2\alpha}(x)$ is an isotropic subspace of $T_xM$. In particular, by Observation \ref{obs:transversality}, we get $\min(p,q) = \dim \g_{2\alpha}$.
\end{proposition}

\begin{remark}
Except in the case of $F_4^{-20}$, we already know that $\min(p,q) \leq \dim \g_{2\alpha}$ thanks to the actions on Einstein Universes mentioned in the introduction.
\end{remark}

Before starting the proof, let us highlight some consequences. We get that $\Sp(1,k)$ cannot act conformally on a compact pseudo-Riemannian manifold of signature $(1,n)$, $n \geq 2$, or $(2,n)$, $n \geq 2$, and that $F_4^{-20}$ cannot act conformally on a compact pseudo-Riemannian manifold whose metric has index strictly less than $7$ (but this is not optimal, see Section \ref{s:f4}). 

In particular, we obtain that if a simple Lie group of real-rank $1$ acts conformally on a compact Lorentzian manifold, then it is locally isomorphic to $\SO(1,k)$, $k \geq 2$ or $\SU(1,k)$, $k \geq 2$. Thus, Theorem \ref{thm:classification_lorentz} is proved modulo a control of $k$ by the dimension of the manifold, which will be proved in Section \ref{s:so1k} and with Lemma \ref{lem:signature_optimal}.

\begin{proof}
Let us assume that $\pi_{x}(\g_{-2\alpha})$ is not isotropic. Then, according to Observation \ref{obs:orthogonality}, $\pi_{x} (\g_{-\alpha} \oplus \g_0 \oplus \g_{\alpha} \oplus \g_{2\alpha})$ would be isotropic. Since $\dim \g_{-\alpha} > \dim \g_{-2\alpha} \geq \min(p,q)$, we necessarily have $\dim \pi_{x}(\g_{-\alpha}) < \dim \g_{-\alpha}$, \textit{i.e.} $\g_{-\alpha} \cap \g_{x} \neq 0$. If $X_{-\alpha} \in \g_{-\alpha} \cap \g_x$, then $[X_{-\alpha},\theta X_{-\alpha}] = |\alpha|^2 B_{\theta}(X_{\alpha},X_{\alpha}) H \in \g_x$ since $\theta X_{-\alpha} \in \s$. So, $\a \subset \g_x$, implying $[\a,\s] = \g_{\alpha} \oplus \g_{2\alpha} \subset \g_x$.

\vspace{.2cm}

Remark that the subspace $\g_{-\alpha} \cap \g_{x}$ is Abelian since for if $X,Y \in \g_{-\alpha} \cap \g_{x}$, then we have $[X,Y] \in \g_{-2\alpha} \cap \g_{x} = 0$. 

\begin{lemma}
\label{lem:lagrangian}
Let $V \subset \g_{-\alpha}$ be a subspace such that $[V,V] = 0$. 
\begin{itemize}
\item If $\g = \su(1,k)$, then $dim V \leq k-1$.
\item If $\g = \sp(1,k)$, then $\dim V \leq 2(k-1)$.
\item If $\g = \f_4^{-20}$, then $\dim V \leq 1$.
\end{itemize}
\end{lemma}

\begin{proof}
In the first case, the bracket $[.,.] : \g_{-\alpha} \wedge \g_{-\alpha} \rightarrow \g_{-2\alpha}$ can be interpreted as a symplectic form on $\R^{2(k-1)}$, whose Lagrangian subspaces are $k-1$ dimensional. In the second case, the bracket is interpreted as a map $\H^{k-1} \wedge \H^{k-1} \rightarrow \Span(i,j,k)$ taking values in the space of purely imaginary quaternions. The component on $i$ of this map is a symplectic form on $\R^{4(k-1)}$, whose isotropic subspaces are at most $2(k-1)$ dimensional. In the case of $\f_4^{-20}$, this bracket is conjugated to the map from $\Oc \times \Oc$ to the space of purely imaginary octonions, given by $(x_1,x_2) \mapsto x_1 \bar{x_2} - x_2 \bar{x_1}$. This map being equivariant under the action of $\mathfrak{spin}(7) \simeq \m$, if $V \neq 0$, we can assume that the unit $1$ belongs to $V$, and it becomes clear that $V$ is the real axis of $\Oc$.
\end{proof}

We now finish the proof case by case.

\begin{itemize}
\item Case $\g = \su(1,k)$. The metric being Lorentzian and since $\dim \pi_{x}(\g_{-\alpha}) \geq 1$, we obtain that $\pi_{x}(\g_0 \oplus \g_{\alpha} \oplus \g_{\alpha}) = 0$. In particular, $\g_0 \subset \g_{x}$.
\item Case $\g = \sp(1,k)$. By the previous lemma, $\dim \pi_{x}(\g_{-\alpha}) \geq 2$, implying $\dim \pi_{x}(\m) \leq 1$ since the metric has index at most $3$. Since $\m \simeq \sp(1) \oplus \sp(k-1)$, it does not admit proper codimension $1$ subalgebras. Thus, $\g_0 \subset \g_{x}$.
\item Case $\g = \f_4^{-20}$. The previous lemma ensures that $\dim \pi_{x} (\g_{-\alpha}) \geq 7$, and since the metric has index at most $7$, we get that $\pi_{x}(\g_0 \oplus \g_{\alpha} \oplus \g_{2\alpha})= 0$. Thus $\g_0 \subset \g_{x}$.
\end{itemize}

Thus, in all cases we have $\g_0 \subset \g_{x}$. It is not difficult to observe that each time, $\ad(\g_0)$ acts irreducibly on $\g_{-\alpha}$. But we have seen that $\g_{-\alpha} \cap \g_{x}$ is a non-trivial proper subspace of $\g_{-\alpha}$, and it is of course invariant under $\ad(\g_0 \cap \g_{x})$. This is our contradiction.
\end{proof}

In the cases of $\SU(1,k)$ and $\Sp(1,k)$, we know that $\min(p,q) = \dim \g_{2\alpha}$. This proposition moreover says that in any compact invariant subset, there is a point $x$ such that $\g_{-2\alpha}(x) \subset T_{x} M$ is a maximally isotropic subspace (compare with Theorem 1 of \cite{bader_nevo}). This fact will be the starting point for the analysis of the orbit of $x$ and the action of the group in a neighborhood of this orbit in Section \ref{s:conformal_flatness}. Precisely, we will derive conformal flatness of a neighborhood of $x$ by considering the action of the isotropy $G_{x}$. We will use repeatedly the following

\begin{observation}
\label{obs:linearizability}
Assume that there exists $x \in M$ such that $\a \oplus \g_{\alpha} \oplus \g_{2\alpha} \subset \g_x$. Then, the conformal vector field $H \in \a$ is locally linearizable near $x$ and if $\phi^t$ denotes the corresponding conformal flow, $T_x \phi^t$ is an hyperbolic one parameter subgroup of $\CO(T_xM,g_x)$. Moreover, for any $X \in \g_{\alpha} \oplus \g_{2\alpha}$, if $f := e^X$, then $T_xf$ is a unipotent element of $\SO(T_xM,g_x)$.
\end{observation}

\begin{proof}
The arguments are based on general properties of automorphisms of Cartan geometries, which are central in the study of pseudo-Riemannian conformal maps.

Let $H \in \a$ and $E \in \g_{\alpha}$ for instance, the other case being exactly similar. Let us note $F := \theta E$. Then, $(H,E,F)$ generates a Lie algebra $\s_{\alpha} < \g$ isomorphic to $\sl(2,\R)$. The vector fields $H$ and $E$ vanish at $x$, and we are in a situation similar to \cite{article4}, Section 4.2.1., the difference being that we are not necessarily in Lorentzian signature. By the same argument based on the horizontality of the curvature form, we obtain a Lie algebra embedding $\rho : \s_{\alpha} \rightarrow \so(p+1,q+1)$, such that there exists $v \in \R^{p+1,q+1} \setminus \{0\}$ isotropic which is an eigenvector for both $\rho(H)$ and $\rho(E)$ but not for $\rho(F)$, and such that $\rho(H)$ and $\rho(E)$ are the respective \textit{holonomies} of $H$ and $E$ at some point $\hx$ over $x$ in the Cartan bundle.

Let us note $\p < \so(p+1,q+1)$ the stabilizer of the line $\R.v$. Since $\rho$ is a finite dimensional representation of $\sl(2,\R)$, which splits into irreducible ones, we see that $\rho(H)$ must be an $\R$-split element of $\p$ and that $\rho(E)$ is a nilpotent element of $\p$. 

We now use the following

\begin{proposition*}[\cite{frances_localdynamics}, Prop. 4.2]
Let $X \in \Kill(M,[g])$ fixing a point $x \in M$. Then, $X$ is locally linearizable near $x$ if and only if its holonomy $X_h \in \p$ is linear.
\end{proposition*}

Recall that $\p$ can be interpreted as the Lie algebra of the conformal group of $\R^{p,q}$, \textit{i.e.} $\p \simeq \co(p,q) \ltimes \R^{p+q}$. An element of $\p$ is said to be ``linear'' if it belongs to $\co(p,q)$ up to conjugacy. Moreover, we can see that the differential $T_x \phi_X^{t}$ of the flow of $X$ is conjugated to the quotient action of $\Ad(\e^{tX_h})$ on $\so(p+1,q+1) / \p$ (the vector field being linearizable or not).

We can directly adapt the proof of \cite{article3}, Lemma 5.5, to the general pseudo-Riemannian setting to conclude that any $\R$-split element in $\p$ is linear. So, this proves that $H$ is locally lienarizable near $x$, and that the differential at $x$ of its flow is $\R$-split. At last, since the holonomy $E_h$ of $E$ is nilpotent, we get that $T_x \phi_E^t$ is a unipotent one-parameter subgroup of $\SO(T_xM)$.
\end{proof}

\section{Compact minimal subsets for Lorentzian conformal actions of $\SO(1,k)$}
\label{s:so1k}

The aim of this section is to explain how to observe directly that if a Lie group $G$ locally isomorphic to $\SO(1,k)$ acts conformally on a compact Lorentzian manifold $(M,g)$ of dimension $n \geq 3$, then $k \leq n$. So, we will assume that $k \geq 4$. The restricted root-space decomposition is $\g = \a \oplus \m \oplus \g_{\pm \alpha}$, with $\m \simeq \so(k-1)$ and $\dim \g_{\pm \alpha} = k-1$, and $\ad(\m)$ acts on $\g_{\pm \alpha}$ via the standard representation of $\so(k-1)$.

\begin{proposition}
Let $x$ be a point given by Proposition \ref{prop:virtual_isotropy}. Then, its $G$-orbit is either
\begin{itemize}
\item A fixed point of $G$;
\item A sphere $\S^{k-1} \simeq G/P$, where $P$ is the stabilizer of an isotropic line of $\R^{1,k}$. The metric $g$ induces the standard conformal Riemannian structure on $\S^{k-1}$.
\item A degenerate orbit of dimension $k$ fibering over $\S^{k-1}$ with $1$-dimensional fiber.
\end{itemize}
\end{proposition}

\begin{remark}
It can be proved that in the third case, the orbit is in fact a \textit{trivial circle bundle} over $\S^{k-1}$. Indeed, the arguments of Section 3.2 of \cite{article4} are directly transposable in this situation: By using Pesin Theory, we can prove that every fiber is a light-like periodic orbit of some hyperbolic one parameter subgroup in $\SO(1,k)$. Thus, by Remark \ref{rem:orbit_closure}, we obtain that any compact minimal $G$-invariant subset of $M$ is one of these three compact $G$-orbits.
\end{remark}

\begin{proof}
As the proof of Obervation \ref{obs:transversality} shows, if $\g_{-\alpha} \cap \g_{x} \neq 0$, then $\g_{x} = \g$ and $x$ is a fixed point of $G$. So, let us assume that $\g_{-\alpha} \cap \g_{x} = 0$. Since $\dim \g_{-\alpha} = k-1 \geq 3$, the subspace $\pi_{x}(\g_{-\alpha})$ cannot be isotropic. By Observation \ref{obs:orthogonality}, we get that $\pi_{x}(\g_0 \oplus \g_{\alpha})$ is isotropic, and hence has dimension at most $1$. In particular, $\g_0 \cap \g_{x}$ is a subalgebra of $\g_0$ with codimension at most $1$. Since $\g_0 \simeq \R \oplus \so(k-1)$ and $k \geq 4$, we obtain that $\g_0 \cap \g_{x} =\m$ or $\g_0 \subset \g_{x}$. In both cases, $\ad(\g_0 \cap \g_{x})$ acts irreducibly on $\g_{\pm \alpha}$. Since $\dim \pi_{x}(\g_{\alpha}) \leq 1$, $\g_{x} \cap \g_{\alpha}$ is a non-trivial subspace of $\g_{\alpha}$, which is $\ad(\g_0 \cap \g_{x})$-invariant. Thus, $\g_{\alpha} \subset \g_{x}$. Finally, we have two possibilities:
\begin{itemize}
\item $\g_{x} = \a \oplus \m \oplus \g_{\alpha}$, or
\item $\g_{x} = \m \oplus \g_{\alpha}$.
\end{itemize}
In both cases, $\m \simeq \so(k-1)$ is in $\g_{x}$. Let $M_G < G$ be the connected Lie subgroup corresponding to $\m$. The restriction of $q_{x}$ to $\pi_{x}(\g_{-\alpha})$ cannot be zero and is $\Ad(M_G)$-invariant. This determines up to a positive factor $q_{x}|_{\pi_{x}(\g_{-\alpha})}$. In the first case, we obtain that $G / G_{x}$ is the $k-1$-sphere and the orbit is Riemannian (necessarily the standard Riemannian structure). In the second case, $\pi_x(\g)$ is degenerate, with kernel $\pi_x(\a)$, and the quadratic form induced on $\pi_x(\g) / \pi_x(\a)$ is the same as the previous case.
\end{proof}

Assume that a Lie group locally isomorphic to $\SO(1,k)$, $k \geq 4$, acts conformally on a compact Lorentzian manifold of dimension $n \geq 3$. \begin{itemize}
\item If there exists a fixed point $x$, then the isotropy representation gives an inclusion $\SO(1,k) \hookrightarrow \SO(1,n-1)$. Thus, $k \leq n-1$.
\item If the group admits a Riemannian spherical orbit, then this orbit cannot be open in $M$, implying that it has dimension at most $n-1$. So, we obtain $k-1 \leq n-1$. 
\item If the group admits an orbit of the third type, this orbit cannot be open in $M$ since it is degenerate, implying that it has dimension at most $n-1$. Thus, $k \leq n-1$.
\end{itemize}

In any case, we obtain that $k \leq \dim M$.

\section{Conformal flatness of Pseudo-Riemannian manifolds in optimal index}
\label{s:conformal_flatness}

Let $k \geq 2$ and $(M,g)$ be a compact pseudo-Riemannian manifold of signature $(p,q)$, with $p+q \geq 3$, on which a Lie group $G$ locally isomorphic to $\SU(1,k)$ or $\Sp(1,k)$ acts conformally. We assume that the index $\min(p,q)$ is minimal with this property, \textit{i.e.} the manifold is Lorentzian if $G \simeq_{\text{loc}} \SU(1,k)$; and has signature $(3,n)$, $n \geq 3$, when $G \simeq_{\text{loc}} \Sp(1,k)$. In this section, we prove conformal flatness of $(M,g)$ in a unified way.

\vspace{.2cm}

The strategy is the following: First, we prove that for any $x \in M$, the closure of its orbit $\bar{G.x}$ contains an orbit $G.y$ which is, up to conformal covers and quotients, the Einstein space of same signature exhibited in the introduction. Then, we prove that such orbits are contained in a conformally flat open subset $U$. Thus, there is $g \in G$ such that $g.x \in U$, proving that $g^{-1}U$ is a conformally flat neighborhood of $x$.

\subsection{Minimal compact subsets}

\begin{lemma}
\label{lem:signature_optimal}
Assume that $\g \neq \f_4^{-20}$. Let $x$ be a point where the conclusions of Proposition \ref{prop:virtual_isotropy} are true. Then, the isotropy of $x$ has the form
\begin{equation*}
\g_x = \a \oplus (\m \cap \g_x) \oplus \g_{\alpha} \oplus \g_{2\alpha}.
\end{equation*}
and $\m \cap \g_x$ has codimension $\dim \g_{2\alpha}$ in $\m$. Moreover, the orbit is non-degenerate and we have the orthogonality relations:
\begin{equation*}
\begin{array}{cccccc}
T_x(G.x) = & \g_{-2\alpha}(x) & \oplus & \g_{-\alpha}(x) & \oplus & \m(x) \\
                 &  \text{Isotropic} & \perp & Euc. & \perp & \text{Isotropic}
\end{array}
\end{equation*}
At last, we also have $\m(x) = \m^1(x)$.
\end{lemma}

In the case $\g = \su(1,k)$, we get that the orbit of any point given by Proposition \ref{prop:virtual_isotropy} is $2k$-dimensional, proving that $2k \leq \dim M$ and finishing the proof of Theorem \ref{thm:classification_lorentz}.

\begin{remark}
When $k \geq 3$, it can be seen that $\m \cap \g_x = \m^2$ (essentially because $\m^2$ will not admit subalgebras of codimension $\dim \g_{2\alpha}$). In this situation, the conformal class on $G.x$ induced by the ambient metric is locally homothetic to the one of the Einstein Universe considered in the introduction. In particular, when $G$ has finite center (\textit{i.e.} when $G \neq \tilde{\SU}(1,k)$), $G.x$ is compact and conformally covered by a finite cover of the Einstein Universe of same signature. By Remark \ref{rem:orbit_closure}, we obtain that any minimal compact $G$-invariant subset of $M$ is such a compact orbit.
\end{remark}

Let us note $d = \dim \g_{2\alpha}$, \textit{i.e.} $d=1$ when $\g = \su(1,k)$ and $d=3$ when $\g = \sp(1,k)$.

\begin{proof}
What we have done previously proves that $\pi_x (\g_{-2\alpha})$ is a maximally isotropic subspace of $(\g / \g_x,q_x)$. Moreover, we know that $\pi_x(\g_{-\alpha}) \neq 0$. In fact, we can say more. We use the following formula that can be observed for instance with the matrix presentation of $\su(1,k)$ in $\gl_{k+1}(\C)$ and of $\sp(1,k)$ in $\gl_{k+1}(\H)$.

\begin{lemma}
Let $X \in \g_{-\alpha}$ and $Z \in \m^1$. Define $Y = [Z,X] \in \g_{-\alpha}$. Then, 
\begin{equation*}
[[\theta Y,X],X] = B_{\theta}(X,X) Y.
\end{equation*}
%
\end{lemma}

Take $X_{-\alpha} \in \g_{-\alpha} \cap \g_x$. Then, by the previous formula (and $\ad(\s)$-invariance of $\g_x$), for any $Z \in \m^1$ we have $[Z,X_{-\alpha}] \in \g_x$. It can be easily observed that $[X_{-\alpha},[Z,X_{-\alpha}]] \neq 0$ unless $X_{-\alpha} = 0$ or $Z = 0$. Thus, if we had $\g_{-\alpha} \cap \g_x \neq 0$, then we would deduce $\g_{-2\alpha} \cap \g_x \neq 0$, contradicting our previous observations.

\vspace{.2cm}

Thus, $\g_{-\alpha} \cap \g_x = 0$. In particular $\dim \pi_x(\g_{-\alpha}) = \dim \g_{-\alpha} > \dim \g_{2\alpha}$, and $\pi_x(\g_{-\alpha})$ cannot be totally isotropic. By Observation \ref{obs:orthogonality}, this space must be orthogonal to $\pi_x(\g_{-2\alpha})$ which is maximally isotropic. Consequently, $\pi_x(\g_{-\alpha})$ is necessarily positive definite. 

\vspace{.2cm}

Moreover, $\pi_x(\g_{\alpha} \oplus \g_{2\alpha})$ is isotropic, but also orthogonal to $\pi_x(\g_{-2\alpha})$. The latter being maximally isotropic, we get $\pi_x(\g_{\alpha} \oplus \g_{2\alpha}) = 0$, meaning that $\g_{\alpha} \oplus \g_{2\alpha} \subset \g_x$.

\vspace{.2cm}

At last, $\pi_x(\g_0)$ is isotropic and orthogonal to $\pi_x(\g_{-\alpha})$. We claim that $\a \subset \g_x$. Indeed, for any non-zero $X_{-2\alpha}$ and consider $A := [\theta X_{-2\alpha}, X_{-2\alpha}] \in \a \setminus 0$. Then, $b_x(\pi_x(X_{-2\alpha}),\pi_x(A)) = -b_x(\pi_x(A),\pi_x(X_{-2 \alpha}))$ because $\theta X_{-2\alpha} \in \s$ and $\ad(\theta X_{-2\alpha})$ is nilpotent (a unipotent linear conformal endomorphism is automatically isometric). So, $\pi_x(A)$ is an isotropic vector orthogonal to $\pi_x(\g_{-2\alpha})$, and we must have $A \in \g_x$.

\vspace{.2cm}

We claim now that the restriction of the metric to the orbit $G.x$ is non-degenerate. To see this, consider $K_x \subset \g / \g_x$ the kernel of $q_x$. Since $\Ad(S)$ acts conformally on $\g / \g_x$, it must leave $K_x$ invariant. In particular, $K_x$ is $\ad(\a)$-invariant and $K_x = (K_x \cap \pi_x(\g_0)) \oplus (K_x \cap \pi_x(\g_{-\alpha})) \oplus (K_x \cap \pi_x(\g_{-2\alpha}))$. Because $\pi_x(\g_{-2\alpha}) \oplus (K_x \cap (\pi_x(\g_0 \oplus \g_{-\alpha})))$ is isotropic, we must have $K_x \cap (\pi_x(\g_0 \oplus \g_{-\alpha}))=0$. Finally, if $X_{-2\alpha} \in K_x \cap \g_{-2\alpha}$ then, $\pi_x(\ad(\g_{\alpha}) X_{-2\alpha}) \subset \pi_x(\g_{-\alpha}) \cap K_x = 0$. This implies that $X_{-2\alpha} =0$ for if not we would have $[X_{-2\alpha},\g_{\alpha}] = \g_{-\alpha}$, and then $\pi_x(\g_{-\alpha}) \subset K_x$. Thus, $K_x = 0$.

\vspace{.2cm}

In particular, $\pi_x(\g_{-2\alpha} \oplus \g_0) = \pi_x(\g_{-\alpha})^{\perp}$ is also non-degenerate. It is at most $2d$-dimensional and contains the $d$-dimensional isotropic subspace $\pi_x(\g_{-2\alpha})$. Necessarily, it has dimension signature $(d,d)$, meaning that $\g_0 \cap \g_x$ has codimension $d$ in $\g_0$.

\vspace{.2cm}

Finally, the Lie algebra of the stabilizer decomposes into:
\begin{equation*}
\g_x = \a \oplus (\m \cap \g_x) \oplus \g_{\alpha} \oplus \g_{2\alpha}.
\end{equation*}
Recall that $\m = \m^1 \oplus \m^2$, with $\m^1$ and $\m^2$ described in Section \ref{ss:general_observations}, and note that $\dim \m^1 = d$. When $k \geq 3$, $\su(k-1)$ admits no non-trivial subalgebra of codimension $1$ and $\sp(k-1)$ admits no non-trivial subalgebra of codimension at most $3$ (see for instance \cite{bohm_kerr}). Thus, $\m^2 \cap \g_x = \m^2$ when $k \geq 3$, \textit{i.e.} $\m \cap \g_x = \m^2$. 

\vspace{.2cm}

Let us focus on the case $k=2$. If $\g = \su(1,2)$, then $\m^2 = 0$ and we have $\m \cap \g_x = 0$. 

If $\g= \sp(1,2)$, then $\m^1 \simeq \m^2 \simeq \sp(1)$. Note $\h = \m \cap \g_x$. It has dimension $3$. If we note $\h^1$ and $\h^2$ its projections on $\m^1$ and $\m^2$, then $\h^i$ is a subalgebra of $\m^i$. As such, it has dimension $0$, $1$ or $3$.

We claim that $\h^2 \neq 0$. Indeed, if not we would have $\m^1 \subset \g_x$. For all $X^1 \in \m^1$, since $X^1$ is elliptic, $\ad(X^1)$ is skew-symmetric with respect to $b_x$. And because $\m^1$ and $\m^2$ commute, we obtain that for all $X^2 \in \m^2$ and $X_{-2\alpha} \in \g_{-2\alpha}$, $b_x(\pi_x(X^2),\pi_x([X^1,X_{-2\alpha}])) = 0$. Since $[\m^1,\g_{-2\alpha}] = \g_{-2\alpha}$, we would have $\pi_x(\g_{-2\alpha}) \perp \pi_x(\m)$, contradicting the fact that $\pi_x(\m)$ is isotropic.

If $\h^2$ was a line in $\m^2$, then $\h^1 = \m^1$ for if not we would have $\dim \h \leq 2$. Then, we would have $\h = \{(X^1,\psi(X^1)), \ X^1 \in \m^1\}$, with $\psi : \m^1 \rightarrow \m^2$ a non-trivial Lie algebra homomorphism. By simpleness, it would be injective, contradicting $\dim \h^2 = 1$.
	
So, we must have $\h^2 = \m^2$, and $\m \cap \g_x = \{(\varphi(X^2),X^2), \ X^2 \in \m^2\}$ where $\varphi : \m^2 \rightarrow \m^1$ is a Lie algebra homomorphism. By simpleness, $\varphi$ is either trivial or is an isomorphism. In both cases, $\m^1 \cap \g_x = 0$.
\end{proof}

\subsection{Conformal flatness near the orbit}

Let $x$ be a point where the conclusions of Proposition \ref{prop:virtual_isotropy} are valid. Let $H \in \a$ be such that $\alpha(H) = 1$ and denote by $\phi^t$ the one-parameter subgroup it generates, which fixes $x$.

\begin{proposition}
\label{prop:linearizability}
The flow of $H$ is locally linearizable near $x$. Precisely, there are $U \subset M$ and $\mathcal{U} \subset T_xM$ open neighborhoods of $x$ and $0$ respectively, which are $\phi^t$ and $T_x\phi^t$ invariant for all $t \geq 0$ respectively, and a diffeomorphism $\psi : \mathcal{U} \rightarrow U$ such that for all $t \geq 0$, $\psi \circ T_x \phi^t = \phi^t \circ \psi$. Moreover, in a basis adapted to the decomposition $T_x M = \m(x) \oplus \g_{-\alpha}(x) \oplus (T_x(G.x))^{\perp} \oplus \g_{-2\alpha}(x)$, 
\begin{equation*}
T_x \phi^t =
\begin{pmatrix}
\id & & & \\
    & e^{-t}\id & & \\
    & & e^{-t} \id & \\
    & & & e^{-2t} \id
\end{pmatrix}
.
\end{equation*}
\end{proposition}

Remark that the fixed points of $\phi^t$ in $U$ form a $d$-dimensional submanifold, which necessary locally coincides with the orbit of the subgroup of $G$ corresponding to $\m^1$, since the latter commutes with with $\a$. Thus, this set of fixed point is a compact, maximally isotropic submanifold on which the dynamics of $\phi^t$ is the same everywhere.

\begin{proof}
The linearizability of $\phi^t$ directly follows from Observation \ref{obs:linearizability}. We simply have to analyse its differential $T_x \phi^t$, which is $\R$-split by the same observation. But its action is clear in restriction to $T_x (G.x)$: it acts identically on $\m(x)$, by homothety of ratio $e^{-t}$ on $\g_{-\alpha}(x)$ and by homothety of ratio $e^{-2t}$ on $\g_{-2\alpha}(x)$. In particular, since $\g_{-\alpha}(x)$ is positive definite, we have $e^t T_x \phi^t \in \SO(T_xM)$. The latter necessarily preserves the positive definite subspace $(T_x (G.x))^{\perp}$, where it induces an $\R$-split one-parameter group of orthogonal linear maps, \textit{i.e.} it acts trivially on it.
\end{proof}

We have now enough information to derive conformal flatness of $g$ in a neighborhood of $x$. Let $\partial_1, \ldots, \partial_n$ be the coordinate frame corresponding to Proposition \ref{prop:linearizability}. Let us define the distributions $\Delta_0 = \Span(\partial_1,\ldots,\partial_d)$, $\Delta_1 = \Span(\partial_{d+1},\ldots,\partial_{n-d})$ and $\Delta_2 = \Span(\partial_{n-d+1},\ldots,\partial_n)$. In particular, $\m^1(x) = \Delta_0(x)$, $\g_{-\alpha}(x) \subset \Delta_1(x)$ and $\g_{-2\alpha}(x) = \Delta_2(x)$. The action of $\phi^t$ on these distributions is very clear:
\begin{itemize}
\item If $1 \leq i \leq d$, then $(\phi^t)_* \partial_i = \partial_i$,
\item If $d+1 \leq i \leq n-d$, then $(\phi^t)_* \partial_i = e^{-t} \partial_i$,
\item If $n-d+1 \leq i \leq n$, then $(\phi^t)_* \partial_i = e^{-2t} \partial_i$.
\end{itemize}

For all $i$, we note $\lambda_i \in \{0,1,2\}$ such that $(\phi^t)_* \partial_i = e^{-\lambda_i t} \partial_i$. Introduce $\|.\|_y$ an arbitrary Riemannian norm on $U$. We can characterize these distributions via the contraction rate of $\phi^t$ via the next elementary observation. 

\begin{lemma}
\label{lem:contraction_rates}
Let $y \in U$ and $v \in T_yM$. If $\|(\phi^t)_*v\| \rightarrow 0$, then $v \in \Delta_1(y) \oplus \Delta_2(y)$. If $\|e^t (\phi^t)_*v\| \rightarrow 0 $, then $v \in \Delta_2(y)$. If $\|e^{2t} (\phi^t)_*v\| \rightarrow 0$, then $v = 0$.
\end{lemma}

Let $W$ be the $(3,1)$-Weyl tensor of $(M,g)$. It admits the same symmetries as the $(3,1)$-Riemann curvature tensor and is conformally invariant. Following the arguments of Frances in the Lorentzian setting (\cite{frances_causal_fields}, Section 3.3), using the previous fact, we can easily deduce that for all $1 \leq i,j,k \leq n$:
\begin{itemize}
\item If $\lambda_i + \lambda_j + \lambda_k = 1$, then $W_y(\partial_i,\partial_j,\partial_k) \in \Delta_1(y) \oplus \Delta_2(y)$;
\item If $\lambda_i + \lambda_j + \lambda_k = 2$, then $W_y(\partial_i,\partial_j,\partial_k) \in \Delta_2(y)$;
\item If $\lambda_i + \lambda_j + \lambda_k \geq 3$, then $W_y(\partial_i,\partial_j,\partial_k) = 0$.
\end{itemize}

\subsubsection{Vanishing of the conformal curvature on the orbit $G.x$}
\label{sss:weyl_vanishing_pointwise}

Consider $u,v,w \in T_x M$, and let $u_i, v_i, w_i$ be their respective component on $\Delta_i(x)$, for $i = 0,1,2$. For any $i,j,k \in \{0,1,2\}$ such that $i+j+k \geq 3$, by Lemma \ref{lem:contraction_rates}, we have $W_x(u_i,v_i,w_i) = 0$. The idea is now to use the tangential action of unipotent elements in the isotropy to propagate the vanishing of the Weyl tensor. Their actions in restriction to the tangent space of the orbit is clear since it is conjugate to $\Ad(G_x)$ acting on $\g / \g_x$. Transversely, we have:

\begin{lemma}
\label{lem:unipotent}
Let $X \in \g_{\alpha} \oplus \g_{2\alpha}$ and $f := e^X \in G_x$. Then, $T_x f$ acts identically on $(T_x(G.x))^{\perp}$.
\end{lemma}

\begin{proof}
We use Observation \ref{obs:linearizability}, which guaranties that $T_xf$ is a unipotent element of $\SO(T_xM)$. It preserves $T_x(G.x)$, which is non-degenerate with maximal index, so it preserves the definite positive subspace $T_x(G.x)^{\perp}$. Its restriction to this Euclidean subspace being unipotent, it must be trivial and we are done.
\end{proof}

Now, let $u_2 = (X_{-2\alpha})_x \in \Delta_2(x)$, and $v_1,w_1 \in \Delta_1(x)$. Let $f$ be any element in the isotropy of $x$ of the form $f= e^{X_{2\alpha}}$. Then,
\begin{align*}
0 = f_* W_x(u_2,v_1,w_1) = W_x(f_* u_2,v_1,w_1) =  W_x(([X_{2\alpha},X_{-2\alpha}])_x,v_1,w_1).
\end{align*}
Indeed, $f_* (X_{-2\alpha})_x = (\Ad(e^{X_{2\alpha}})X_{-2\alpha})_x = (X_{-2\alpha})_x + ([X_{2\alpha},X_{-2\alpha}])_x$. Similarly we observe that $f_*$ acts trivially on $\g_{-\alpha}(x)$, implying with Lemma \ref{lem:unipotent} that $f_* v_1 = v_1$ and $f_* w_1 = w_1$. Since $\m^1 \subset [\g_{-2\alpha},\g_{2\alpha}]$ and $\Delta_0(x) = \m^1(x)$, we obtain that $W_x(\Delta_0,\Delta_1,\Delta_1) = 0$. Of course, the same reasoning works for all permutations of $(\Delta_0,\Delta_1,\Delta_1)$. 

\vspace{.2cm}

Similarly, for any $u_2 = (X_{-2\alpha})_x \in \Delta_2(x)$, $v_0 \in \Delta_0(x)$ and $w_1 \in \Delta_1(x)$, we have
\begin{equation*}
0 = f_* W_x(u_2,v_0,w_1) =  W_x(([X_{-2\alpha},X_{2\alpha}])_x,v_0,w_1),
\end{equation*}
and it follows that $W_x(\Delta_0,\Delta_0,\Delta_1) = 0$, together with all permutations of $(\Delta_0,\Delta_0,\Delta_1)$.

\vspace{.2cm}

It can easily be observed that $\m^1 \subset [\g_{\alpha},\g_{-\alpha}]$, and for any $X_{\alpha}$ and $u_1 = (X_{-\alpha})_x$, $v_0,w_0 \in \Delta_0(x)$, if we note $f=  e^{X_{\alpha}}$ we get
\begin{equation*}
0 = f_* W_x(u_1,v_0,w_0) = W_x([X_{\alpha},X_{-\alpha}]_x,v_0,w_0).
\end{equation*}
So, we obtain $W_x(\Delta_0,\Delta_0,\Delta_0) = 0$. 

\vspace{.2cm}

Finally, let $u_0,v_0,z_0 \in \Delta_0(x)$ and $w_2 \in \Delta_2(x)$. We know that $W_x(u_0,v_0,w_2) \in \Delta_2(x) = (\Delta_1(x) \oplus \Delta_2(x))^{\perp}$. Moreover, using the symmetry of $W$, we have
\begin{equation*}
g_x(W_x(u_0,v_0,w_2),z_0) = - g_x(W_x(u_0,v_0,z_0),w_2) = 0.
\end{equation*}
Thus, $W_x(u_0,v_0,w_2) \in \Delta_0(x)^{\perp}$, proving that it is orthogonal to the whole tangent space $T_xM$, \textit{i.e.} $W_x(\Delta_0,\Delta_0,\Delta_2) = 0$, together with every permutation of the indices. 

\vspace{.2cm}

So, we have obtained that $W_x = 0$, and immediately the Weyl tensor also vanishes in restriction to the orbit $G.x$. 

\subsubsection{Conformal flatness near $G.x$}

We note $\mathcal{Z} := U \cap \{y \in M \ | \ \phi^t(y) = y\}$. Let $y$ be some point in $U$. Since $\phi^t(y)$ converges when $t \to + \infty$ to a point in $\mathcal{Z}$, and because $\mathcal{Z} \subset G.x$, we have
\begin{equation*}
\forall y \in U, \ \phi^t(y) \rightarrow y_{\infty}, \text{ with } W_{y_{\infty}} = 0.
\end{equation*}
Consequently, since $\| T_y \phi^t \|$ is bounded (the norm refers to our arbitrary Riemannian metric), we have that for any $u,v,w \in T_yM$, $\phi^t_* W_x(u,v,w) \rightarrow 0$. By Lemma \ref{lem:contraction_rates}, this proves that the Weyl tensor takes values in $\mathcal{H} := \Delta_1 \oplus \Delta_2$. Moreover, we already know that $W(\mathcal{H},\mathcal{H},\mathcal{H}) = 0$ everywhere in $U$. 

\begin{fact}
The distribution $\mathcal{H}$ is maximally degenerate everywhere in $U$ and $\Ker(g|_{\mathcal{H}}) = \Delta_2$.
\end{fact}

By maximally degenerate subspace, we mean the orthogonal of some maximally isotropic subspace.

\begin{proof}
Let us first see that $\Delta_1$ is positive definite all over $U$. This is already the case at $x$ since $\Delta_1(x) = \g_{-\alpha}(x) \oplus (T_x(G.x))^{\perp}$. Let $M_G^1 < G$ be the connected subgroup corresponding to $\m^1$. Since $\Ad(M_G^1)$ preserves $\g_{-\alpha}$, for all $g \in M_G^1$, $g_* \Delta_1 = \Delta_1$. In particular, since $\mathcal{Z}$ is contained in the $M_G^1$-orbit of $x$, we get that $\Delta_1$ is positive definite in restriction to $\mathcal{Z}$ and then, in some open neighborhood of $\mathcal{Z}$. Since $(\phi^t)_* \Delta_1 = \Delta_1$ and since for all $y \in U$, $\phi^t(y)$ converges to some point of $\mathcal{Z}$, we obtain that $\Delta_1$ is everywhere Euclidean.

Now, let $\lambda(y,t)$ be the conformal distortion of $\phi^t$ with respect to $g$. By definition, $\lambda(y,t) g_y(\partial_{d+1},\partial_{d+1}) = g_{\phi^t(y)}((\phi^t)_* \partial_{d+1}(y),(\phi^t)_* \partial_{d+1}(y)) = e^{-2t} g_{\phi^t(y)}(\partial_{d+1},\partial_{d+1})$. Then we obtain 
\begin{equation*}
e^{2t} \lambda(y,t) \rightarrow \frac{g_{y_{\infty}}(\partial_{d+1},\partial_{d+1})}{g_y(\partial_{d+1},\partial_{d+1})} =: C >0.
\end{equation*}
On the other hand, for all ${d+1} \leq i \leq n-d$ and $n-d+1 \leq j \leq n$, we have $\lambda(y,t) g_y(\partial_i,\partial_j) = \e^{-3t} g_{\phi^t(y)}(\partial_i,\partial_j)$. So, we obtain that $C g_y(\partial_i,\partial_j) = 0$, proving that $\Delta_1(y) \perp \Delta_2(y)$. Similarly, we obtain that $\Delta_2(y) \perp \Delta_2(y)$.
\end{proof}

\begin{lemma}
\label{lem:two_distributions}
Let $V = \R^{p,q}$ denote the standard pseudo-Euclidean space, and let $T$ be a $(3,1)$-tensor on $V$ having the same symmetries as a $(3,1)$-curvature tensor. Let $V_1,V_2 \subset V$ be two subspaces such that
\begin{enumerate}
\item $V_1 \cap V_2$ is non-degenerate;
\item $V = V_1 + V_2$.
\end{enumerate}
If $T(V_i,V_i,V_i) = 0$ for $i =1,2$ and if $\Ima T \subset V_1 \cap V_2$, then $T=0$.
\end{lemma}

\begin{proof}
The proof is similar to the one of Lemma 9 of \cite{article3} presented in Lorentzian signature where $V_1$ and $V_2$ are degenerate hyperplanes. Let us note $F = V_1 \cap V_2$ and choose $V_1'$ and $V_2'$ such that $V_i = F \oplus V_i'$. We then have to prove that for all $u_1,v_1 \in V_1'$, $u_2,v_2 \in V_2'$ and $w \in F$, the following terms are zero:
\begin{equation*}
\begin{array}{l}
T(u_1,v_1,u_2), \ T(u_1,u_2,v_1), \  \\
T(u_2,v_2,u_1), \ T(u_2,u_1,v_2), \  \\
T(u_1,u_2,w),   \ T(u_2,w,u_1),   \ T(w,u_1,u_2).
\end{array}
\end{equation*}
By hypothesis, all these terms belong to $F$. And for all $w' \in F$, we see that all these terms are orthogonal to $w'$: for instance, $<T(u_1,u_2,v_1),w'> = <T(v_1,w',u_1),u_2>$ and $T(v_1,w',u_1) = 0$ since $v_1,w',u_1 \in V_1$. This finishes the proof since $F \cap F^{\perp} = 0$.
\end{proof}
 
The distribution $\mathcal{H}$ will play the role of $V_1$ in every $T_yM$, with $y \in U$. The idea is now to twist it by some suitable unipotent element in the isotropy of $x$ to obtain another distribution $\mathcal{H'} := g^* \mathcal{H}$ satisfying the same properties as $\mathcal{H}$ since the isotropy acts conformally.

\vspace{.2cm}

Let us choose $X_{\alpha}$ non-zero and consider the action of $g = e^{X_{\alpha}}$ on $\Delta_2(x) = \Ker \mathcal{H}(x)$. It preserves $T_x(G.x)$ on which its action is conjugate to the linear action on $\g / \g_x$ induced by $\Ad(g)$. Thus, for all $(X_{-2\alpha})_x \in \Delta_2(x)$, we have
\begin{equation*}
T_x g (X_{-2\alpha})_x = \underbrace{(X_{-2\alpha})_x}_{\in \Delta_2(x)} + \underbrace{[X_{\alpha},X_{-2\alpha}]_x}_{\in \Delta_1(x)} + \frac{1}{2} \underbrace{([X_{\alpha},[X_{\alpha},X_{-2\alpha}]])_x}_{\in \Delta_0(x)}
\end{equation*}
If $T_x g (X_{-2\alpha})_x \in \Delta_2(x)$, then the second and the third terms of the right hand side are zero. But $[X_{\alpha},X_{-2\alpha}]_x = 0$ implies $[X_{\alpha},X_{-2\alpha}]_x = 0$ since $\g_{-\alpha} \cap \g_x = 0$, which implies $X_{-2\alpha} = 0$, since $X_{\alpha} \neq 0$ (as it can be observed in matrix representations of $\g$). Consequently, $(g_* \Delta_2(x)) \cap \Delta_2(x) = 0$, and necessarily $(g^* \Delta_2) \cap \Delta_2 = 0$ on a neighborhood of $x$. 

Consider now the distribution $\mathcal{H}' = g^* \mathcal{H}$. It is maximally degenerate with kernel $g^* \Delta_2$ and $\mathcal{H} \cap \mathcal{H}'$ is positive definite. Since the Weyl tensor is conformally invariant, we have $W_y(g^* \mathcal{H},g^* \mathcal{H},g^* \mathcal{H}) = g^{-1}_* W_{g.y}(\mathcal{H},\mathcal{H},\mathcal{H}) = 0$ and $\Ima W \subset g^* \mathcal{H}$. Thus, the pair $(\mathcal{H},g^* \mathcal{H})$ satisfies the hypothesis of Lemma \ref{lem:two_distributions} in every $T_yM$, for $y$ in some suitable neighborhood of $x$, proving that such neighborhood must be conformally flat.
 
\section{Pseudo-Riemannian conformal actions of $F_4^{-20}$}
\label{s:f4}

Let $(M,g)$ be a compact pseudo-Riemannian manifold of signature $(p,q)$ on which $F_4^{-20}$ acts conformally. By Proposition \ref{prop:lower_bound}, we know that $\min(p,q) \geq 7$, since the restricted root-space associated to $2\alpha$ has dimension $7$. Contrarily to the cases of conformal actions of $\SU(1,k)$  and $\Sp(1,k)$, this lower bound is not sharp: we will see that the optimal index is $9$, a realization of which being observed with the action of $F_4^{-20}$ on $\Ein^{9,15}$ given in the introduction. So, we have to see that there are no conformal actions when the metric has index $7$ or $8$ and we will treat each situation independently.

\vspace{.2cm}

Throughout this section, $G$ denotes $F_4^{-20}$ and is assumed to act conformally on $(M,g)$, compact with signature $(p,q)$ whose index is optimal, and $\g = \a \oplus \m \oplus \g_{\pm \alpha} \oplus \g_{\pm 2\alpha}$ denotes a restricted root-space decomposition. We still note $\s = \a \oplus \g_{\alpha} \oplus \g_{2\alpha}$ and $S < F_4^{-20}$ the corresponding subgroup. We fix $x$ a point given by Proposition \ref{prop:virtual_isotropy}.  Recall that $\g_x \cap \g_{-2\alpha} = 0$. Moreover, since the only Abelian subspaces of $\g_{-\alpha}$ are lines (Lemma \ref{lem:lagrangian}), we must have $\dim \g_{-\alpha} \cap \g_x \leq 1$.

\vspace{.2cm}

Recall that $\m \simeq \so(7)$, and that $\dim \g_{\pm \alpha} = 8$ and $\dim \g_{\pm 2\alpha} = 7$. The adjoint action $\ad(\m)$ on $\g_{\pm \alpha}$ is conjugate to the spin representation, and the one on $\g_{\pm 2\alpha}$ to the regular representation. We will use that the only non-trivial sugalgebras of $\so(7)$ whose codimension is less or equal than $8$ are isomorphic to $\so(6)$ and $\g_2$ (see for instance \cite{bohm_kerr}), and that $G_2$-subgroups of $\Spin(7)$ are stabilizers of vectors in the spin representation (see \cite{varadajan}).

\subsection{Index $7$}
\label{ss:index7}

We assume here that $\min(p,q) = 7$. Then, by Proposition \ref{prop:lower_bound} we know that $\pi_x(\g_{-2\alpha})$ is a maximal isotropic subspace. We distinguish two cases:
\begin{itemize}
\item Either $\exists X_{-\alpha}, X_{-2\alpha}$ such that $b_x(\pi_x(X_{-2\alpha}),\pi_x(X_{-\alpha}))$,
\item Or $\pi_x(\g_{-\alpha})$ is positive definite and orthogonal to $\pi_x(\g_{-2\alpha})$.
\end{itemize}
By Observation \ref{obs:orthogonality}, if we are in the first case, then $\pi_x(\g_0)$ is isotropic and orthogonal to the maximally isotropic subspace $\pi_x(\g_{-2\alpha})$. So, we obtain $\g_0 \subset \g_x$. Since $\ad$ induces an irreducible representation of $\m$ on $\g_{-\alpha}$, and since $\g_x \cap \g_{-\alpha} \neq \g_{-\alpha}$, we obtain $\g_x \cap \g_{-\alpha} = 0$. Thus, $\pi_x(\g_{-\alpha})$ is an $8$-dimensional space, which must be isotropic by Observation \ref{obs:orthogonality}. This is not possible.

If we are in the second case, then $\pi_x(\g_0 \oplus \g_{\alpha} \oplus \g_{2\alpha})$ is isotropic, and consequently at most $7$-dimensional. By the same argument used in the proof of Lemma \ref{lem:signature_optimal}, we obtain that $\a \subset \g_x$. Since $\g_x$ is $\ad(\s)$-invariant, we also get $\g_{\alpha} \oplus \g_{2\alpha} \subset \g_x$.

The Lie subgroups of codimension $\leq 7$ of $\Spin(7)$ are $\Spin(7)$ itself, $\Spin(6)$ and $G_2$. We now distinguish three cases:
\begin{enumerate}
\item $\m \subset \g_x$;
\item $\m \cap \g_x$ has codimension $6$ in $\m$;
\item $\m \cap \g_x$ has codimension $7$ in $\m$.
\end{enumerate}

In case (1), we obtain $\g_x \cap \g_{-\alpha} = 0$. If we choose $X_{\alpha}$ and $X_{-2\alpha}$ such that $[X_{\alpha},X_{-2\alpha}] \neq 0$, since $\pi_x(\g_{-\alpha})$ is positive definite, we have
\begin{align*}
0 & \neq b_x(\pi_x([X_{\alpha},X_{-2\alpha}]),\pi_x([X_{\alpha},X_{-2\alpha}])) \\
  & = -b_x(\pi_x(X_{-2\alpha}),\pi_x([X_{\alpha},[X_{\alpha},X_{-2\alpha}]])).
\end{align*}
This is not possible since $\pi_x([X_{\alpha},[X_{\alpha},X_{-2\alpha}]]) = 0$.

\vspace{.2cm}

In case (2), $\m \cap \g_x$ is a Lie subalgebra of $\m$ which has codimension $6$ in $\m$. Thus, we obtain that $T_x(G.x) = \g_{-2\alpha}(x) \oplus \g_{-\alpha}(x) \oplus \m(x)$, with $\dim \m(x) = 6$. The subspace $\g_{-2\alpha}(x) \oplus \m(x)$ is $13$-dimensional and contains a $7$-dimensional isotropic subspace. Necessarily, it is degenerate and its kernel $K_x \neq 0$ is $\ad(\a)$ invariant. Thus, $K_x = (K_x \cap \g_{-2\alpha}(x)) \oplus (K_x \cap \m(x))$. Moreover $\g_{-2\alpha}(x) \oplus (K_x \cap \m(x))$ is isotropic, implying $K_x \cap \m(x) = 0$, and $\m(x) \oplus (K_x \cap \g_{-2\alpha}(x))$ isotropic implies that $K_x$ is a line included in $\g_{-2\alpha}(x)$.

Since $\g_{-\alpha}(x)$ is positive definite and orthogonal to $\g_{-2\alpha}(x) \oplus \m(x)$, $K_x$ is in fact the kernel of $T_x(G.x)$. As such, it is invariant by the isotropy, implying that $[K_x,\g_x] \subset K_x$. But this is not possible since for non-zero $X_{-2\alpha}$, we have $[X_{-2\alpha},\g_{\alpha}] = \g_{-\alpha}$.

\vspace{.2cm}

In case (3), we immediately get that $\g_{\alpha} \oplus \g_{2\alpha} \subset \g_x$. Arguments similar to the previous ones give that the orbit is non-degenerate. Precisely, $\pi_x(\g_{-2\alpha}) \oplus \pi_x(\m)$ has signature $(7,7)$ and is orthogonal to $\pi_x(\g_{-\alpha})$ which is positive definite and has dimension $8$ or $7$.

Thus, the homogeneous space $N = G/G_x$ is endowed with a conformally $G$-invariant non-degenerate metric $h$. We claim that such a metric does not exist. We can prove it by observing that $h$ must be conformally flat. Indeed, if $F_4^{-20}$ acts on a conformally flat pseudo-Riemannian manifold of signature $(p',q')$ with $p'+q' \geq 3$, then we can derive a Lie algebra embedding $\f_4^{-20} \hookrightarrow \so(p'+1,q'+1)$. Thus, if $(N,h)$ is conformally flat, then we obtain a faithful representation of $\f_4^{-20}$ on a vector space of dimension $23$ or $24$, which is absurd since the smallest degree of a faithful representation of $\f_4$ is $26$.

Conformal flatness of $h$ can be easily proved by using the arguments of Section \ref{sss:weyl_vanishing_pointwise}. Indeed, if $H \in \a$ is such that $\alpha(H) = 1$, then the flow of $H$ has the same form - this fact being immediate since $N$ is $G$-homogeneous. We also have $[\g_{2\alpha},\g_{-2\alpha}] = \g_0$ and $[\g_{\alpha},\g_{-\alpha}] = \g_0$. Thus, we can apply the same reasonning to conclude that the Weyl tensor of $(N,h)$ vanishes at some point $x_0 \in N$, and we directly obtain that $(N,h)$ is conformally flat by homogeneity.

\vspace{.2cm}

Finally, all the cases lead to a contradiction and we conclude that the metric cannot have index $7$. Remark that this section has proved the following

\begin{lemma}
\label{lem:index7}
There does not exist a proper Lie subgroup $H<G$ such that 
\begin{itemize}
\item $G/H$ carries a conformally $G$-invariant (eventually degenerate) metric whose isotropic subspaces are at most $7$-dimensional,
\item $\Ad(S) \h \subset \h$ and $\Ad(S)$ acts conformally on $\g/\h$.
\end{itemize}
\end{lemma}

\subsection{Index $8$}

We assume now that the metric has index $8$. We still have $\dim \pi_x(\g_{-2\alpha}) = 7$ and $\dim \pi_x(\g_{-\alpha}) = 7$ or $8$.

\begin{lemma}
The subspace $\g_{-2\alpha}(x)$ is isotropic.
\end{lemma}

\begin{proof}
Assume that $\pi_x(\g_{-2\alpha})$ is not isotropic. By Observation \ref{obs:orthogonality}, we obtain that $\pi_x(\g_{-\alpha} \oplus \g_0 \oplus \g_{\alpha} \oplus \g_{2\alpha})$ is isotropic and orthogonal to $\pi_x(\g_{-2\alpha})$. Since $\dim \pi_x(\g_{-\alpha}) \geq 7$, we get that $\dim \pi_x(\g_0) \leq 1$. Since $\m \simeq \so(7)$ does not admit codimension $1$ subalgebras, we get $\m \subset \g_x$. By irreducibility of $\ad(\m)|_{\g_{-\alpha}}$, we get $\g_{-\alpha} \cap \g_x = 0$, implying that $\g_x = \p := \g_0 \oplus \g_{-\alpha} \oplus \g_{2\alpha}$. At last, since $\pi_x(\g_{-2\alpha}) \perp \pi_x(\g_{-\alpha})$, we obtain that $\pi_x(\g_{-2\alpha})$ is positive definite.

\vspace{.2cm}

If $P < G$ is the proper parabolic subgroup corresponding to $\p$, we have proved that the orbit of $x$ has the form $G/P$ and carries a conformally $G$-invariant maximally degenerate metric. But it is clear that such a metric exist. So, the question is to prove that we cannot embed conformally and $G$-equivariantly this homogeneous space into $(M,g)$. This is why we now turn to more geometric considerations, even though our goal is to prove that such a situation \textit{does not} occur.

\vspace{.2cm}

Let $H \in \a$ be such that $\alpha(H) = 1$ and let $\phi^t$ be the conformal flow it generates. By Lemma \ref{lem:unipotent}, we get that $\phi^t$ is linearizable near $x$ and that $\{T_x \phi^t\} < \CO(T_xM, g_x)$ is an hyperbolic one-parameter subgroup. Moreover, since $\g_{-2\alpha}(x)$ is positive definite, we obtain that $e^{2t} T_x \phi^t \in \SO(T_xM,g_x)$.

\begin{sublemma}
\label{sublem:hyperbolic_isometric_element}
Let $X \in \so(p,q)$ be an hyperbolic element, and $V \subset \R^{p,q}$ be a maximally isotropic, $X$-invariant subspace such that $X|_V = \lambda \id$, $\lambda \neq 0$. Then, there exists $V'$ maximally isotropic such that $V \oplus V'$ has signature $(p,p)$, and $X$ has the form
\begin{equation*}
X = 
\begin{pmatrix}
\lambda \id & & \\
 & 0 & \\
 & & -\lambda \id
\end{pmatrix},
\end{equation*}
with respect to $\R^{p,q}  = V \oplus (V \oplus V')^{\perp} \oplus V'$.
\end{sublemma}

\begin{proof}
We have that $X$ preserves $V^{\perp}$ and induces on it an $\R$-split linear map. By semi-simplicity, let $W$ be $X$-invariant and such that $V^{\perp} = V \oplus W$. Since $V$ is maximally isotropic, $W$ is positive definite, and $X|_W \in \so(W)$ implies $X|_{W} = 0$. At last, let $V'$ be $X$-invariant and such that $\R^{p,q} = V^{\perp} \oplus V'$. Let $v' \in V'$ be an eigenvector of $X$ with eigenvalue $\lambda'$. Then, for all $v \in V$,
\begin{equation*}
\lambda <v,v'> = <Xv,v'> = - \lambda' <v,v'>.
\end{equation*}
Since $v' \notin V^{\perp}$, we can choose $v$ such that $<v,v'> \neq 0$, proving $\lambda' = -\lambda$. Thus, $X$ acts homothetically and non-trivially on $V'$, implying that $V'$ is isotropic and $V' \perp W$.
\end{proof}

In our situation, $e^{2t}T_x \phi^t$ acts homothetically on $V_1 = \g_{-\alpha}(x)$, with ratio $e^t$. This proves that there are $V_2,V_3 \subset T_xM$, with $\g_{-2\alpha}(x) \subset V_2$ positive definite, $V_3$ maximally isotropic, and $T_xM = V_1 \oplus V_2 \oplus V_3$ and
\begin{equation*}
T_x\phi^t =
\begin{pmatrix}
e^{-t}\id & & \\
 & e^{-2t}\id & \\
 & & e^{-3t}\id &
\end{pmatrix}.
\end{equation*}
Let $W$ denote the Weyl tensor of $(M,g)$. Then, the arguments of Section \ref{sss:weyl_vanishing_pointwise} can be naturally adapted to conclude that for all $u,v,w \in T_xM$, with respective components $u_i,v_i,w_i$ on $V_i$, we have $W_x(u,v,w) = W_x(u_1,v_1,w_1) \in V_3$ (adapt the Lemma \ref{lem:contraction_rates}). For all $X_{-2\alpha}$ and $X_{\alpha}$, if $f = e^{X_{\alpha}}$, we have for all $v_1,w_1$,
\begin{equation*}
0 = f_* W_x((X_{-2\alpha})_x,v_1,w_1) = W_x(f_* (X_{-2\alpha})_x,v_1,w_1) = W_x([X_{\alpha},X_{-2\alpha}]_x,v_1,w_1),
\end{equation*}
the first equality following from $f_* v_1 = v_1$ and $f_* w_1 = w_1$ (this is immediate because $v_1,w_1 \in \g_{-\alpha}(x)$). This implies $W_x(V_1,V_1,V_1)= 0$ since $\g_{-\alpha} = [\g_{\alpha},\g_{-2\alpha}]$.

\vspace{.2cm}

Thus, $W_x = 0$. Let $y$ be a point in the linearization neighborhood of $x$. Since $\phi^t(y) \rightarrow x$ when $t \to +\infty$, we get that for all $u,v,w \in T_yM$, $\|(\phi^t)_* W_y(u,v,w)\| = o(e^{-3t})$ for any norm $\|.\|$ on this neighborhood. By an argument similar to Lemma \ref{lem:contraction_rates}, we get $W_y = 0$. Thus, a neighborhood of $x$ is conformally flat, and this is a contradiction since there does not exist an embedding of $\f_4^{-20}$ into some $\so(9,N)$.
\end{proof}

\begin{lemma}
The subspace $\g_{-\alpha}(x)$ is orthogonal to $\g_{-2\alpha}(x)$. Consequently, it is contained in a Lorentzian subspace of $T_xM$.
\end{lemma}

\begin{proof}
Assume the contrary. Then, by Observation \ref{obs:orthogonality}, $\pi_x(\g_{-\alpha} \oplus \g_0 \oplus \g_{\alpha} \oplus \g_{2\alpha})$ is isotropic. Since $\dim \pi_x(\g_{-\alpha}) \geq 7$, we obtain that $\g_0 \oplus \g_{\alpha} \oplus \g_{2\alpha} \subset \g_x$, and then that $\g_{-\alpha} \cap \g_x = 0$. Then, we obtain that $\g(x) = \g_{-2\alpha}(x) \oplus \g_{-\alpha}(x)$, these two subspace being isotropic and having dimension $7$ and $8$ respectively. Necessarily, $\g(x)$ is degenerate and its kernel is a $1$-dimensional subspace of $\g_{-\alpha}(x)$. This is a contradiction since $\g_0 \subset \g_x$ cannot preserve this line.
\end{proof}

By the previous Lemma, $\pi_x(\g_{-\alpha})$ cannot be isotropic. Thus, $\pi_x(\p)$ is isotropic and orthogonal to it.

\begin{lemma}
We have $\a \subset \g_x$, and by $\ad \s$-invariance, we get $\s \subset \g_x$.
\end{lemma}

\begin{proof}
Assume that $\a \cap \g_x = 0$. We can use the argument of the proof of Lemma \ref{lem:signature_optimal} to conclude that $\a(x)$ is an isotropic line orthogonal to $\g_{-2\alpha}(x)$, which is isotropic and $7$-dimensional. For any $X \in \g_{\alpha} \oplus \g_{2\alpha}$, by Obervation \ref{obs:orthogonality}, $X_x$ is isotropic and orthogonal to $\g_{-2\alpha}(x) \oplus \a(x)$. The latter being maximally isotropic, we obtain $X_x = 0$, \textit{i.e.} $\g_{\alpha} \oplus \g_{2\alpha} \subset \g_x$. Always by Observation \ref{obs:orthogonality}, $\g_{-\alpha}(x)$ is orthogonal to a maximally isotropic subspace, so it must be Euclidean.

Consider $V := \g / (\a \oplus \g_x)$. Since $\a$ is included in the kernel of $q_x$, the latter induces a well-defined quadratic form $q_0$ on $V$. Moreover, since $T_x(G.x) \subset \a(x)^{\perp}$ and the metric has index $8$, by definition of $q_x$ we get that the isotropic subspaces of $(V,q_0)$ are at most $7$-dimensional. At last, since $\a \oplus \g_x$ is a Lie subalgebra stable under the action of $\ad(\s)$, we obtain a conformal action of $\s$ on $(V,q_0)$. This contradicts Lemma \ref{lem:index7}.
\end{proof}

Finally, we get the orthogonal decomposition $\g(x) = \g_{-2\alpha}(x) \oplus \g_{-\alpha}(x) \oplus \m(x)$, with $\g_{-2\alpha}(x)$ isotropic and $7$ dimensional, $\m(x)$ isotropic, and $\g_{-\alpha}(x)$ subLorentzian. Since $\dim \m(x) \leq 8$, we distinguish three possibilities: (a) $\m \subset \g_x$, (b) $\m \cap \g_x \simeq \so(6)$ and (c) $\m \cap \g_x \simeq \g_2$. Each of which is proved to be impossible.	

\vspace{.2cm}

In case (a), we get $\g_{-\alpha} \cap \g_x = 0$ and $\g_{-\alpha}(x)$ must be Euclidean since the only non-trivial quadratic form on $\R^{8}$ invariant by $\Spin(7)$ is positive definite. We can apply the same argument as in the case $\min(p,q) = 7$ to conclude that this is not possible. 

In case (b), we get $T_x(G.x) = \g_{-2\alpha}(x) \oplus \g_{-\alpha}(x) \oplus \m(x)$ with $\g_{-2\alpha}(x)$ isotropic and $7$-dimensional, $\m(x)$ isotropic and $6$-dimensional and $\g_{-2\alpha}(x) \perp \g_{-\alpha}(x) \perp \m(x)$. Moreover, since $\m \cap \g_x$ cannot preserve a $7$-dimensional subspace of $\g_{-\alpha}$, we also have $\g_{-\alpha} \cap \g_x = 0$. Necessarily, $T_x (G.x)$ is degenerate and the kernel meets $\g_{-2\alpha}(x)$. But if $(X_{-2\alpha})_x$ is in the kernel, then so is $[X_{\alpha},X_{-2\alpha}]_x \in \g_{-\alpha}(x)$ for any $X_{\alpha}$. And the kernel cannot meet $\g_{-\alpha}(x)$ for if not it would yield a direction $X_{-\alpha}^0 \in \g_{-\alpha}$ preserved by $\ad(\m \cap \g_x)$. 
So, we get $[X_{\alpha},X_{-2\alpha}]= 0$ for any $X_{\alpha}$, implying $X_{-2\alpha} = 0$, which is a contradiction.

In case (c), $\g_{-2\alpha}(x)$ and $\m(x)$ are isotropic and $7$-dimensional. Moreover, the subspace $\g_{-2\alpha}(x) \oplus \m(x)$ is non-degenerate with signature $(7,7)$. Indeed, if $K_x$ is its kernel, then $K_x \cap \g_{-2\alpha}(x)$ and $K_x \cap \m(x)$ are at most $1$-dimensional. The stabilizer $\g_x$ contains a subalgebra $\g_2 \subset \m \simeq \so(7)$. The action of $\ad(\m)$ on $\g_{-2\alpha}$ and $\m$ being clear, we see that this $\g_2$ factor cannot preserve a line in $\g_{-2\alpha}(x)$ nor $\m(x)$. 
Thus, $K_x = 0$.

Finally, we distinguish two sub-possibilities: (i) $\g_{-\alpha}(x)$ is non-degenerate, \textit{i.e.} it is Euclidean or Lorentzian ; (ii) $\g_{-\alpha}(x)$ is a degenerate, non-negative subspace with a $1$-dimensional kernel. 

In case (c)(i), the orbit $G.x$ is non-degenerate. If $H \in \a$, by analyzing the action of $\phi_H^t$ near $x$ on $G.x$, we can conclude that $(M,g)$ induces a conformally flat metric on $G.x$, leading to the same contradiction as case (3) of Section \ref{ss:index7}.

In case (c)(ii), we have $\g_{-\alpha} \cap \g_x = 0$, so that $\g_x = (\a \oplus (\m \cap \g_x)) \ltimes (\g_{\alpha} \oplus \g_{2\alpha})$ with $\m \cap \g_x \simeq \g_2$, and there is $X_{-\alpha}^0 \neq 0$ such that the kernel of the restriction of $g_x$ to $T_x(G.x)$ corresponds to the line $\R X_{-\alpha}^0 + \g_x$ in $\g / \g_x$. Since the isotropy $G_x$ has to preserve this kernel, $\R X_{-\alpha}^0 + \g_x$ must be a subalgebra of $\g$, which we note $\h$. Then, let $H < G$ be the corresponding connected Lie subgroup. Since $X_{-\alpha}^0$ gives the direction of the kernel in $\g / \g_x$, we obtain a well defined non-degenerate quadratic form $q_0$ on $\g / \h$ such that, if $\pi : G/G_x \rightarrow G/H$ is the natural projection, then $\pi_* : (\g / \g_x,q_x) \rightarrow (\g / \h,q_0)$ is isometric and $\ad(\g_x) < \co(\g / \h, q_0)$. Let $x_0 = H \in G/H$ and $\phi^t := e^{tX_{-\alpha}^0}$, so that $\phi^t$ is a conformal flow $G/G_x$, commuting with $\pi$. Thus, it is also a conformal flow of $G/H$, proving that $\ad(\h) < \co(\g / \h, q_0)$. Finally, we obtain that $G/H$ is endowed with a conformally $G$-invariant non-degenerate metric of index $7$, which is not possible by Lemma \ref{lem:index7}.

\bibliographystyle{amsalpha}
\bibliography{references_article_semi_simples.bib}
\nocite{*}
\end{document}